\documentclass{amsart}

\usepackage{amsmath,amssymb,amsthm,mathrsfs,stackrel}
\usepackage{color}

\newcommand{\BC}{{\mathbb C}}\newcommand{\BD}{{\mathbb D}}

\newcommand{\BN}{{\mathbb N}}

\newcommand{\BR}{{\mathbb R}}
\newcommand{\BT}{{\mathbb T}}

\newcommand{\cD}{{\mathcal D}}
\newcommand{\cE}{{\mathcal E}}
\newcommand{\cH}{{\mathcal H}}

\newcommand{\cK}{{\mathcal K}}\newcommand{\cL}{{\mathcal L}}
\newcommand{\cM}{{\mathcal M}}

\newcommand{\cU}{{\mathcal U}}\newcommand{\cV}{{\mathcal V}}
\newcommand{\cW}{{\mathcal W}}\newcommand{\cX}{{\mathcal X}}
\newcommand{\cY}{{\mathcal Y}}\newcommand{\cZ}{{\mathcal Z}}

\newcommand{\bS}{{\mathbf S}}

\newcommand{\fR}{{\mathfrak R}}

\newcommand{\wtilD}{\widetilde{D}}

\newcommand{\wtilM}{\widetilde{M}}\newcommand{\wtilN}{\widetilde{N}}

\newcommand{\wtilQ}{\widetilde{Q}}\newcommand{\wtilR}{\widetilde{R}}

\newcommand{\wtilX}{\widetilde{X}}
\newcommand{\wtilY}{\widetilde{Y}}
\newcommand{\al}{\alpha}

\newcommand{\de}{\delta}
\newcommand{\vep}{\varepsilon}

\newcommand{\la}{\lambda}

\newcommand{\vph}{\varphi}
\newcommand{\om}{\omega}

\newcommand{\ran}{\textup{Ran\,}}

\newcommand{\im}{\textup{Im\,}}
\newcommand{\re}{\textup{Re\,}}
\newcommand{\kr}{\textup{Ker\,}}
\newcommand{\diag}{\textup{diag\,}}

\newcommand{\mat}[2]{\ensuremath{\left[\begin{array}{#1}#2\end{array} \right]}}

\newcommand{\sbm}[1]{\left[\begin{smallmatrix} #1\end{smallmatrix}\right]}
\newcommand{\ov}[1]{{\overline{#1}}}
\newcommand{\un}[1]{{\underline{#1}}}
\newcommand{\inn}[2]{\ensuremath{\langle #1,#2 \rangle}}

\newcommand{\wtil}[1]{{\widetilde{#1}}}

\newcommand{\half}{\frac{1}{2}}
\newcommand{\ands}{\quad\mbox{and}\quad}
\newcommand{\ons}{\mbox{ on }}

\newtheorem{theorem}{Theorem}[section]
\newtheorem{corollary}[theorem]{Corollary}
\newtheorem{lemma}[theorem]{Lemma}
\newtheorem{proposition}[theorem]{Proposition}

\theoremstyle{definition}

\newtheorem{example}[theorem]{Example}

\newcommand{\oran}{\overline{\textup{Ran}}\,}

\newcommand{\precI}{\stackrel{{\scriptscriptstyle \infty\ \ }}{\prec}\!\!}

\newcommand{\simI}{\stackrel{{\scriptscriptstyle \infty}}{\sim}}

\title[A pre-order and an equivalence relation on Schur class functions]{A pre-order and an equivalence relation on Schur class functions and their invariance under linear fractional transformations}

\author{S. ter Horst}

\address{%
Unit for BMI, North West University\\
Private Bag X6001-209, Potchefstroom 2520, South Africa}

\email{sanne.terhorst@nwu.ac.za}


\date{}

\begin{document}
\maketitle

\begin{abstract}
Motivated by work of Yu.L. Shmul'yan a pre-order and an equivalence relation on the set of operator-valued Schur class functions are introduced and the behavior of Redheffer linear fractional transformations (LFTs) with respect to these relations is studied. In particular, it is shown that Redheffer LFTs preserve the equivalence relation, but not necessarily the pre-order. The latter does occur under some additional assumptions on the coefficients in the Redheffer LFT.
\end{abstract}

\setcounter{section}{-1}
\section{Introduction}
\setcounter{equation}{0}

In a 1980 paper \cite{S80} Yu.L. Shmul'yan introduced a pre-order relation on the set of Hilbert space contractions, and showed that linear fractional maps of Redheffer type, as initially studied by R.M. Redheffer in \cite{R60,R62}, preserve this pre-order.

To be more precise, let $\cH_1$ and $\cH_2$ be Hilbert spaces. With $\cL_1(\cH_1,\cH_2)$ we denote the set of contractions from $\cH_1$ to $\cH_2$, and, given $C\in \cL_1(\cH_1,\cH_2)$, we write $D_C$ for the {\em defect operator} $D_C=(I-C^*C)^{\half}$ and $\cD_C$ for the {\em defect space} $\cD_C=\oran D_C$ associated with $C$. Then for $A,B\in\cL_1(\cH_1,\cH_2)$ Shmul'yan writes $A\prec B$ if there exists a bounded operator $X$ from $\cD_B$ into $\cD_{B^*}$ such that $A-B=D_{B^*}XD_B$, and proves that this defines a pre-order on $\cL_1(\cH_1,\cH_2)$. (Actually, in \cite{S80} this is denoted by $B\prec A$, however, more in line with other pre-orders on $\cL_1(\cH_1,\cH_2)$, the order is reversed in \cite{KSS91}, and we will adopt the notation of \cite{KSS91} here.) Section \ref{S:BallPreOrder} contains a detailed discussion of the pre-order $\prec$ and the corresponding equivalence relation $\sim$. To give an idea, the set of strict contractions form an equivalence class, and strict contractions dominate all other contractions, and if $A\in\cL_1(\cH_1,\cH_2)$ is an isometry or co-isometry, then $A$ forms an equivalence class by itself and $A$ is dominated by no other contraction than itself, see Lemma \ref{L:ConesClasses} below.

Now let $\cK_1$ and $\cK_2$ be two additional Hilbert spaces and let
\begin{equation}\label{M}
M=\mat{cc}{M_{11}&M_{12}\\ M_{21} &M_{22}}:\mat{c}{\cH_2\\ \cK_1}\to \mat{cc}{\cH_1\\ \cK_2}
\end{equation}
be a contraction in $\cL_1(\cH_2\oplus\cK_1, \cH_1\oplus\cK_2)$. Then the Redheffer (linear fractional) map associated with $M$ is defined as
\[
\fR_M[A]=M_{22}+M_{21}A(I-M_{11}A)^{-1}M_{12}\quad (A\in\cL_1(\cH_1,\cH_2),\ I-M_{11}A\mbox{ invertible}).
\]
Shmul'yan proved the following result.

\begin{theorem}[Theorem 7 in \cite{S80}]\label{T:Shmul'yan}
Let $M\in \cL_1(\cH_2\oplus\cK_1, \cH_1\oplus\cK_2)$, as in \eqref{M}, and $A,B\in\cL_1(\cH_1,\cH_2)$ such that $I-M_{11}A$ and $I-M_{11}B$ are invertible. Assume $A\prec B$. Then $\fR_M[A]\prec \fR_M[B]$.
\end{theorem}

The aim of the present paper is to investigate generalizations and extensions of this result to the case where the operators are replaced by Schur class functions, i.e., contractive analytic functions. For separable Hilbert spaces $\cU$ and $\cY$ we write $S(\cU,\cY)$ for the {\em Schur class} consisting of $\cL_1(\cU,\cY)$-valued analytic functions on the open unit disc $\BD=\{\la \colon |\la|<1\}$.  Recall that a Schur class function $F\in S(\cU,\cY)$ defines a contractive analytic Toeplitz operator $T_F$ between the vector-valued Hardy spaces $H^2_\cU$ and $H^2_\cY$ by
\[
(T_F h)(\la)=F(\la)h(\la)\quad (h\in H^2_\cU,\ \la\in\BD).
\]
(Analytic here means $T_F$ intertwines the forward shift operators on $H^2_\cU$ and $H^2_\cY$.) It is well known that the supremum norm of $F$ on $\BD$ equals the operator norm of $T_F$. This enables us to define a pre-order on $S(\cU,\cY)$, denoted as $\precI$, by defining $F\precI G$ whenever $T_F\prec T_G$. The following theorem gives a characterization of $\precI$ in terms of the elements of $S(\cU,\cY)$.

\begin{theorem}\label{T:SCprecDef}
Let $F,G\in S(\cU,\cY)$. Then $T_F\prec T_G$ if and only if there existence of a bounded operator-valued function $Q$ on $\BD$ such that
\[
F(\la)-G(\la)=D_{G(\la)^*}Q(\la) D_{G(\la)}\quad (\la\in\BD).
\]
\end{theorem}

This result is one of several characterizations of $\precI$ that are proved in Theorem \ref{T:precSC} below. It is immediate from the results on $\prec$ in Section \ref{S:BallPreOrder} that the set of strict Schur class functions, i.e., $F\in S(\cU,\cY)$ with $\|F\|_\infty<1$, forms an equivalence class, which dominates all other equivalence classes, and each inner function, i.e.\ $F\in S(\cU,\cY)$ with $T_F$ isometric, forms an equivalence class on its own which is dominated by no other equivalence class but its own. In Section \ref{S:LaurentPreOrder} we also consider the impact of the pre-order $\precI$ on the boundary behavior of functions in $S(\cU,\cY)$. The equivalence relation defined by $\precI$ will be denoted by $\simI$.

Redheffer maps in this context involve a Schur class function $\Phi \in S(\cE'\oplus \cU,\cE\oplus\cY)$, with $\cE$ and $\cE'$ also separable Hilbert spaces, partitioned as
\begin{equation}\label{PhiDecom}
\Phi(\la)=\mat{cc}{\Phi_{11}(\la)& \Phi_{12}(\la)\\ \Phi_{21}(\la) &\Phi_{22}(\la)}:\mat{c}{\cE'\\ \cU}\to\mat{c}{\cE\\\cY}\quad (\la\in\BD).
\end{equation}
The Redheffer map associated with $\Phi$ is defined as
\begin{equation}\label{IntroRed}
\fR_\Phi[F](\la)=\Phi_{22}(\la)+ \Phi_{21}(\la)F(\la)(I-\Phi_{11}(\la)F(\la))^{-1}\Phi_{12}(\la) \quad (\la\in\BD)
\end{equation}
where $F$ in is $S(\cE,\cE')$ and should satisfy $\|\Phi_{11}(0)F(0)\|<1$. The additional assumption $\|\Phi_{11}(0)F(0)\|<1$ is necessary and sufficient for the inverse to be well defined for each $\la\in\BD$, and it is well known that $\fR_\Phi[F]\in S(\cU,\cY)$ for any $F\in S(\cE,\cE')$ which satisfies this constraint. Redheffer maps of this type play an important role in the theory of metric constrained interpolation and system and control theory, cf., \cite{AD08,FFGK98,F87} and the references given there. Typically in applications $\|\Phi_{11}(0)\|<1$, or even $\Phi_{11}(0)=0$, such that $\fR_\Phi$ is defined on the whole of $S(\cE,\cE')$.

The following theorem is the main result of the present paper, and will be proved in Section \ref{S:Redheffer}.

\begin{theorem}\label{T:RedInvar}
Let $\Phi \in S(\cE'\oplus \cU,\cE\oplus\cY)$ decompose as in \eqref{PhiDecom} and let $F_1,F_2\in S(\cE,\cE')$ with $F_1\simI F_2$. Then $\|\Phi_{11}(0)F_1(0)\|<1$ if and only if $\|\Phi_{11}(0)F_2(0)\|<1$. Moreover,  if $\|\Phi_{11}(0)F_i(0)\|<1$ holds for $i=1,2$, then $\fR_\Phi[F_1]\simI \fR_\Phi[F_2]$.
\end{theorem}

Hence $\fR_\Phi$ preserves the equivalence relation $\simI$. On may now wonder if $\fR_\Phi$ also preserves the pre-order $\precI$. This is in general not the case. An example given by Bakonyi in \cite{B95} proves this. Indeed, the example of \cite{B95} deals with a specific suboptimal scalar Nehari interpolation problem. By the general theory of metric constrained interpolation, cf. \cite{FF90}, it is known that the solutions set of such a problem is given by the range of a Redheffer map $\fR_\Phi$, of the type considered in the present paper, and that there must be a solution $\fR_\Phi[F]$ with $\|\fR_\Phi[F]\|_\infty<1$. Write $\un{0}$ for the constant function whose value is the zero operator. It is shown in \cite{B95} that $\|\fR_\Phi[\un{0}]\|_\infty=1$. Clearly $\un{0}$ is a strict Schur class function. It then follows by the harmonic maximum principle due to A. Biswas \cite{B97} that $F$ cannot be a strict Schur class function. Hence we have $F\prec\un{0}$ and $\fR_\Phi[\un{0}]\precI \fR_\Phi[F]$, and $F$ and $\un{0}$ nor $\fR_\Phi[F]$ and $\fR_\Phi[\un{0}]$ belong to the same equivalence class. See Example 5.6 and Remark 5.11 in \cite{AD08} for a discussion of larger class of examples, which includes the one of \cite{B95}, where this phenomenon occurs. Example \ref{E:nonprecinv} below presents another example where $\fR_\Phi$ does not preserve $\precI$ with $\Phi$ an operator polynomial of degree one.

Note that one can view $\fR_\Phi$ in the context of the Redheffer maps $\fR_M$ of \cite{S80} by taking $M=T_\Phi$, i.e., $M_{ij}=T_{\Phi_{ij}}$, and $A=T_F$. However, the condition $\|\Phi_{11}(0)F(0)\|<1$ does not imply that $I-T_{\Phi_{11}}T_F$ is (boundedly) invertible, and hence $\fR_\Phi$ acts on a larger domain than $\fR_{T_\Phi}$. This observation yields the following result, which can both be viewed as a generalization of Theorem \ref{T:Shmul'yan}, with $\Phi$ and $F$ in \eqref{IntroRed} constant, and a specification of Theorem \ref{T:Shmul'yan}, with $M=T_\Phi$ and $A=T_F$.

\begin{theorem}\label{T:RedInvar2}
Let $\Phi \in S(\cE'\oplus \cU,\cE\oplus\cY)$ decompose as in \eqref{PhiDecom} and $F,G\in S(\cE,\cE')$ such that $F\precI G$. Assume that the function $\|\Phi_{11}(0)F(0)\|<1$ is boundedly invertible on $\BD$. Then $\fR_\Phi[F]\precI \fR_\Phi[G]$. In particular, if $\Phi_{11}$ is a strict Schur class function, then $\fR_\Phi$ maps $S(\cE,\cE')$ into $S(\cU,\cY)$ and preserves the pre-order $\precI$.
\end{theorem}

This result is proved in Section \ref{S:Redheffer}, where we also consider how the behavior of $\fR_\Phi$ with respect to $\precI$ improves even further when $\Phi$ meets additional constraints. In particular, we prove, in Proposition \ref{P:InjectClass}, that $\fR_\Phi$ induces an injective map on the equivalence classes of $S(\cE,\cE')$ in case $\Phi_{12}$ or $\Phi_{21}$ is invertible on $\BD$ with a bounded analytic inverse.

Besides the current introduction, the paper consists of three sections. In Section \ref{S:BallPreOrder} we discuss various definitions and implications of the pre-order $\prec$ and equivalence relation $\sim$ of \cite{S80} and derive relations between the parameters in the different definitions. The specification of $\prec$ and $\sim$ to contractive Toeplitz and Laurent operators, both analytic and nonanalytic, leading to the definition of the pre-order $\precI$ and equivalence relation $\simI$, is the topic of Section \ref{S:LaurentPreOrder}. The properties of Redheffer maps $\fR_\Phi$ in connection with $\precI$ and $\simI$ are investigated in the final section.

We conclude this introduction with some words on notation and terminology. Throughout capital calligraphic letters denote Hilbert spaces (with the exception of $\cL$). With operator and subspace we mean a bounded linear map and a closed linear manifold. In particular, in this paper all operators are bounded and all subspaces closed. Given Hilbert spaces $\cH_1$ and $\cH_2$ we write $\cL(\cH_1,\cH_2)$ for the space of operators from $\cH_1$ to $\cH_2$. We further write $\cL_1(\cH_1,\cH_2)$ and $\cL_1^\circ(\cH_1,\cH_2)$ for the closed, respectively open, unit balls in $\cL(\cH_1,\cH_2)$, i.e., the sets of contractions, respectively strict contractions. In case $\cH_1=\cH_2$, we simply write $\cL(\cH_1)$, $\cL_1(\cH_1)$ and $\cL_1^\circ(\cH_1)$ instead of $\cL(\cH_1,\cH_1)$, $\cL_1(\cH_1,\cH_1)$ and $\cL_1^\circ(\cH_1,\cH_1)$, respectively.

Given $C\in\cL_1(\cH_1,\cH_2)$ we denote the {\em defect operator} and {\em defect space} associated with $C$ by $D_C$, respectively $\cD_C$, that is $D_C$ is the positive square root of $I_{\cH_1}-C^*C$ and $\cD_C$ is the closure of the range of $D_C$. With some abuse of notation, we will view the defect operator $D_C$ either as an operator on $\cH_1$, on $\cD_C$, from $\cH_1$ into $\cD_C$ or from $\cD_C$ into $\cH_1$, always using the symbol $D_C$. The precise meaning will be clear from the context, or otherwise be made explicit. For an operator $A\in\cL(\cH_1,\cH_2)$ the real and imaginary part of $A$ are the self-adjoint operators defined by $\re(A)=\half (A+A^*)$ and $\im(A)=\frac{1}{2i}(A-A^*)$, respectively. The kernel of $A$ is denoted by $\kr A$ and range of $A$ by $\ran A$; the symbol $\oran A$ stands for the closure of $\ran A$. If $A$ is invertible, then $A^{-*}$ is short for $(A^*)^{-1}$, or $(A^{-1})^*$ which amounts to the same operator. Let $B\in\cL(\cH_1)$ be a self-adjoint operator. Then $B\geq 0$ indicates that $B$ is positive semi-definite, i.e., $\inn{Bh}{h}\geq 0$ for each $h\in\cH_1$. An invertible positive semi-definite operator $B$ is called positive definite, and this is indicated by $B>0$. If $C\in\cL(\cH_1)$ is also self-adjoint, then $B\geq C$ (resp.\ $B>C$) is short for $B-C\geq 0$ (resp.\ $B-C>0$). The orthogonal projection on a subspace $\cM$ of $\cH_1$ is denoted by $P_{\cM}$.

Let $\cU$ and $\cY$ denote separable Hilbert spaces. The Banach space of Lebesgue measurable, essentially bounded $\cL(\cU,\cY)$-valued functions on the unit circle $\BT$, together with the essential supremum norm $\|\ \|_\infty$, will be denoted by $L^\infty(\cU,\cY)$.  We write $H^\infty(\cU,\cY)$ for the Banach space of bounded analytic $\cL(\cU,\cY)$-valued functions on the open unit disc $\BD$, with the supremum norm on $\BD$, also denoted by $\|\ \|_\infty$. Note that $H^\infty(\cU,\cY)$ can be viewed as a sub-Banach space $L^\infty(\cU,\cY)$ by taking nontangential limits to a.e.\ point on $\BT$. The closed unit ball of $H^\infty(\cU,\cY)$ is the Schur class $S(\cU,\cY)$. We write $S_0(\cU,\cY)$ for the open unit ball of $H^\infty(\cU,\cY)$, i.e., the set of strict Schur class functions. With  $L^\infty_1(\cU,\cY)$ and $L^\infty_{<1}(\cU,\cY)$ we indicate the closed, respectively open, unit ball of $L^\infty(\cU,\cY)$.

With $K\in L^\infty(\cU,\cY)$ we associate it's Laurent operator $L_K$ mapping $L^2_\cU$ into $L^2_\cY$ and it's Toeplitz operator $T_K$ from $H^2_\cU$ to $H^2_\cY$, which are the operators defined by
\[
(L_K g)(e^{it})=K(e^{it}) g(e^{it})\quad (a.e.\ t\in [0,2\pi)),\quad
T_K=P_{H^2_\cY} L_K|_{H^2_\cU}.
\]
Note that this definition of $T_K$ is equivalent to the one given earlier in this introduction.
It is well known that $\|K\|_\infty=\|L_K\|=\|T_K\|$. A function $F\in S(\cU,\cY)$ is called inner if the nontangential limits of $F$ on $\BT$ are a.e.\ isometries, $*$-inner if the nontangential limits are a.e.\ co-isometries, and two-sided inner if it is both inner and $*$-inner. Note that $F$ is inner if and only if $T_F$ is an isometries, or equivalently $L_F$ is an isometry.

\section{A pre-order on contractive Hilbert space operators}
\setcounter{equation}{0}\label{S:BallPreOrder}

Throughout this section $\cH_1$ and $\cH_2$ are Hilbert spaces.
In \cite{S80} Yu.L. Shmul'yan considered the pre-order relation
on the set of contractions $\cL_1(\cH_1,\cH_2)$ defined in the
following theorem.

\begin{theorem}\label{T:prec}
Let $A,B\in\cL_1(\cH_1,\cH_2)$.
\begin{itemize}
\item[(i)]
The relation $A  \prec B$ defined by one of the
following four equivalent conditions:
\begin{itemize}
\item[(POi)]  $A-B= D_{B^*}X D_B$  for some $X\in\cL(\cD_{B},\cD_{B^*})$;

\item[(POii)] $I-A^*B= D_{B}Y D_B$ for some $Y\in\cL(\cD_B)$;

\item[(POiii)] there exists an $r>0$ such that
$(1-\vep)B+\vep A\in\cL_1(\cH_1,\cH_2)$
for all $\vep\in\BC$ with $|\vep|\leq r$;

\item[(POiv)] there exists an $r>0$ such that
$(1-\vep)B+\vep A\in\cL_1(\cH_1,\cH_2)$
for all $\vep\in\BC$ with $|\vep|=r$;
\end{itemize}
defines a pre-order relation on $\cL_1(\cH_1,\cH_2)$.

\item[(ii)] Assume $A  \prec B$ and let $Y\in\cL(\cD_B)$ be as in (POii). Then
\[
2\,\re (Y)\geq I \ands \kr (2\re(Y)-I)=\{0\}.
\]

\item[(iii)] Assume $A  \prec B$. Then the operators $X$ and $Y$
and the positive constant $r$ from (POi)--(POiv)
can be chosen in such a way that the following relations hold:
\begin{align}
\label{XvsY}
& \|X\|\leq \|Y\|+\sqrt{2\|\re(Y)\|-1},\quad \|Y\|\leq 1+\|X\|,\\
\label{XnYineq}
&\|X\|\leq \frac{2+2\sqrt{r}+r}{2r},\quad \|Y\|\leq \frac{2+r}{2r},\\
\label{rineq}
&  r
\geq \frac{1}{\|Y\|+\sqrt{\|Y\|^2+2\|\re(Y)\|-1}}.
\end{align}
\end{itemize}
\end{theorem}

Shmul'yan \cite{S80} initially studied this pre-order in the form of (POi). In a later paper by Khatskevich, Shmul'yan and Shul'man \cite{KSS91} it was shown that the formulation of $\prec$ in (POi) is equivalent to those in (POii) and (POiii), by extending ideas from real convex analysis to the complex numbers. The paper \cite{KSS91} also discusses the relation between the pre-order of Theorem \ref{T:prec}, the Harnack pre-order from \cite{iS72} and a pre-order defined by Ceausescu in \cite{C76}.

Statements (ii) and (iii) of Theorem \ref{T:prec} are additions to the results of \cite{KSS91}. The relations between the constants appearing in (POi)--(POiv) from statement (iii) are necessary for our specification to contractive analytic Laurent operators in Section \ref{S:LaurentPreOrder} in terms of their symbols, and require a refinement of many of the arguments used in \cite{KSS91}.

Before proving Theorem \ref{T:prec} we first derive a few preliminary results. We start by mentioning the following easily verified identity, which holds for any $A,B\in\cL_1(\cH_1,\cH_2)$, independent of the pre-order relation:
\begin{equation}\label{re(I-A^*B)identity}
2\re(I-A^*B)= D_A^2+D_B^2+(A-B)^*(A-B).
\end{equation}

The following lemma provides a useful reformulation of
conditions (POiii) and (POiv).

\begin{lemma}\label{L:geo}
Let $A$ and $B$ be contractions in $\cL_1(\cH_1,\cH_2)$ and $\vep\in\BC$.
Then $(1-\vep)B+\vep A\in\cL_1(\cH_1,\cH_2)$ holds if and only if
\begin{equation}\label{poscon}
\begin{array}{c}
0\leq(1-2\re (\vep)) D_B^2+2\re(\vep)\, \re(I-A^*B)+\\[.2cm]
\mbox{ }\qquad\qquad\qquad+2\im (\vep)\, \im(I-A^*B)-|\vep|^2(A-B)^*(A-B).
\end{array}
\end{equation}
\end{lemma}

\begin{proof}[\bf Proof.]
The result follows since for each $\vep\in\BC$ we have
\begin{align*}
&I-((1-\vep)B+\vep A)^*((1-\vep)B+\vep A)=\\
&\qquad=I-|1-\vep|^2 B^*B-|\vep|^2 A^*A-2\re((1-\vep)\ov{\vep}A^*B)\\
&\qquad=|1-\vep|^2 D_B^2+|\vep|^2 D_A^2+(1-|1-\vep|^2-|\vep|^2)I
-2\re((1-\vep)\ov{\vep}A^*B)\\
&\qquad=|1-\vep|^2 D_B^2+|\vep|^2 D_A^2-2(|\vep|^2-\re(\vep))I
-2\re(\ov{\vep}-|\vep|^2) \re(A^*B)+\\
&\qquad\qquad\qquad\qquad+2\im(\ov{\vep}-|\vep|^2) \im(A^*B)\\
&\qquad=|1-\vep|^2 D_B^2+|\vep|^2 D_A^2-2(|\vep|^2-\re(\vep))I
+2(|\vep|^2-\re(\vep))\re(A^*B)+\\
&\qquad\qquad\qquad\qquad-2\im(\vep)\im(A^*B)\\
&\qquad=|1-\vep|^2 D_B^2+|\vep|^2 D_A^2-2(|\vep|^2-\re(\vep))\re(I-A^*B)+\\
&\qquad\qquad\qquad\qquad+2\im(\vep)\im(I-A^*B)\\
&\qquad=(|1-\vep|^2 -|\vep|^2)D_B^2+2\re(\vep)\re(I-A^*B)
+2\im(\vep)\im(I-A^*B)+\\
&\qquad\qquad\qquad\qquad-|\vep|^2(A-B)^*(A-B)\\
&\qquad=(1-2\re(\vep))D_B^2+2\re(\vep)\re(I-A^*B)
+2\im(\vep)\im(I-A^*B)+\\
&\qquad\qquad\qquad\qquad-|\vep|^2(A-B)^*(A-B),
\end{align*}
using \eqref{re(I-A^*B)identity} in the last but one identity.
\end{proof}

The following lemma will be of use when analyzing
condition (POii).

\begin{lemma}\label{L:ReImIneq}
Let $C\in \cL(\cH_1)$ and $D\in\cL(\cH_1,\cH_2)$. Then there exists an
operator $Z\in\cL(\cH_2)$ with $C=D^* Z D$  if and only if there exist
$\de_1,\de_2\geq0$ such that
\[
-\de_1 D^*D \leq \re(C) \leq \de_1 D^*D\ands
-\de_2 D^*D \leq \im(C) \leq \de_2 D^*D.
\]
Moreover, if $C=D^* Z D$ for $Z\in\cL(\cH_2)$ then one can take $\|\re(Z)\|=\de_1$ and $\|\im(Z)\|=\de_2$.
\end{lemma}

\begin{proof}[\bf Proof.]
The result follows by observing that $C=D^* Z D$ holds if and only if $\re (C)=D^* \re(Z)D$ and $\im (C)=D^* \im(Z)D$, and then applying the next lemma with $U=\re(C)$, respectively $U=\im(C)$.
\end{proof}

\begin{lemma}\label{L:Basic}
Let $U\in \cL(\cH_1)$ be selfadjoint and $D\in\cL(\cH_1,\cH_2)$. Then there exists an operator $W\in\cL(\cH_2)$ with $U=D^*WD$  if and only if there exists a $\de\geq0$ such that $-\de D^*D \leq U \leq \de D^*D$. Moreover, if $W\in\cL(\cH_2)$ with $U=D^*WD$, then one can take $\de=\|W\|$.
\end{lemma}

\begin{proof}[\bf Proof.]
First assume $-\de D^*D \leq U \leq \de D^*D$ for some $\de\geq0$. By the spectral theorem for selfadjoint operators, see e.g.\ Chapter V in \cite{GGK90}, we can write $U=U_+-U_-$ with $U_+\geq 0$ and $U_-\geq 0$. Then
\[
-\de D^*D\leq -U_- \leq U \leq U_+\leq \de D^*D.
\]
By Douglas' lemma \cite{D66} we find that there exist contractions $V_+,V_-\in\cL(\cH_2,\cH_1)$ with $U_+^\half=\de^\half V_+D$ and $U_-^\half=\de^\half V_-D$. This yields $U_+=U_+^{*\half}U_+^\half=\de D^* V_+^*V_+ D$ and, similarly, $U_-=\de D^*V_-^*V_-D$. Thus the identity $U=D^* W D$ is satisfied with $W=\de(V_+^*V_+-V_-^*V_-)$.

Conversely, assume $U=D^* W D$ for some $W\in\cL(\cH_2)$. Since $U$ is selfadjoint, we may without loss of generality assume $W$ is selfadjoint. Now write $W=W_+-W_-$ with $W_+\geq 0$ and $W_-\geq 0$. Define $U_+=D^*W_+ D$ and $U_-=D^* W_- D$. Then $0\leq U_+=D^*W_+ D\leq\de_+ D^*D$, with $\de_+=\|W_+\|$, and, similarly, $0\leq U_-\leq\de_- D^*D$, with $\de_-=\|W_-\|$. Set $\de=\max\{\de_+,\de_-\}=\|W\|$. Then
\[
-\de D^*D\leq -\de_- D^*D\leq -U_-\leq U\leq U_+\leq \de_+ D^*D\leq \de D^*D.
\]
Hence our claim follows.
\end{proof}

\begin{proof}[\bf Proof of Theorem \ref{T:prec}.]
The proof is split into four parts.

\noindent {\bf Part I.} In the first part we prove statement (ii). Assume $A\prec B$, in the sense of (POii), i.e., $I-A^*B= D_B YD_B$ for some $Y\in\cL(\cD_B)$. Then $\re (I-A^*B)=D_B \re(Y) D_B$. Using \eqref{re(I-A^*B)identity}, we see that
\[
2D_B\re(Y)D_B=D_A^2+D_B^2+(A-B)^*(A-B).
\]
Hence
\begin{equation}\label{ReYpos}
D_B(2\re(Y)-I_{\cD_{B}})D_B=D_A^2+(A-B)^*(A-B)\geq 0.
\end{equation}
Since $\oran D_B=\cD_B$, it follows that $2\re(Y)-I_{\cD_{B}}\geq0$. In order to show $\kr(2\re(Y)-I_{\cD_{B}})=\{0\}$, it suffices to prove for any $x\in\cD_B$ that
\[
0=D_B(2\re(Y)-I_{\cD_{B}})D_Bx=\mat{cc}{D_A& (A-B)^*}\mat{c}{D_A\\ (A-B)}x
\]
implies $x=0$. So assume this identity is satisfied for some $x\in\cD_B$. Then $D_Ax=0$ and $Ax=Bx$. Since $D_Ax=0$, we have $\|Ax\|=\|x\|$. Thus $\|Bx\|=\|x\|$. However, $x\in \cD_B$ implies $\|Bx\|<\|x\|$, unless $x=0$. Thus necessarily $x=0$.\medskip

\noindent {\bf Part II.} Next we show the equivalence  of conditions (POi)--(POiv).

{\bf (POi) $\Rightarrow$ (POii):}
Take $Y=I-X^*B$, with $X$ as in (POi). Then
\begin{align*}
I-A^*B
&=D_B^2-(A^*-B^*)B=D_B^2-D_BX^*D_{B^*}B\\
&=D_B^2-D_BX^*BD_{B}=D_BYD_B.
\end{align*}

{\bf (POii) $\Rightarrow$ (POiii):}
Let $Y\in\cL(\cD_B)$ be as in (POii). Then
\[
\re (I-A^*B)=D_B \re(Y) D_B \ands \im (I-A^*B)= D_B\im(Y) D_B.
\]
Using $(A-B)^*(A-B)\leq 2\re(I-A^*B)-D_B^2$, by \eqref{re(I-A^*B)identity}, and writing $\vep=s e^{i\theta}$ we see that the right hand side in \eqref{poscon} dominates
\begin{align*}
&(1-2\re(\vep)+|\vep|^2)D_B^2+2(\re(\vep)-|\vep|^2)\re(I-A^*B)+\\
&\quad+2\im(\vep)\im(I-A^*B)=\\
&\quad=
D_B((1-2s\cos\theta+s^2)I+(2s\cos\theta-s^2)\re (Y)+2s\sin\theta \im(Y))D_B\\
&\quad=D_B(I+(2(\re(Y)-I)\cos\theta+2\im(Y)\sin\theta)s-(2\re(Y)-I)s^2)D_B.
\end{align*}
By Lemma \ref{L:geo}, (POiii) holds if there exists an $r>0$ such that for all $0\leq s\leq r$ and all $\theta\in[-\pi,\pi]$ we have
\begin{equation}\label{interform}
I+(2(\re(Y)-I)\cos \theta+2\im(Y)\sin\theta)s-(2\re(Y)-I)s^2\geq 0.
\end{equation}

For $\theta\in[-\pi,\pi]$, let $\theta_+\in[-\half\pi,\half\pi]$ be such that $\sin\theta_+=\sin\theta$ and $\cos\theta_+=|\cos \theta|$. Define $\vep_+=s(\sin\theta_++i\cos\theta_+)$. By statement (ii) we have $4\re(Y)-2 I\geq 0$. Thus $2(\re(Y)- I)\geq -2\re(Y)$. Therefore
\[
2(\re(Y)-I)\cos \theta\geq - 2\re(Y)|\cos \theta| =- 2\re(Y) \cos\theta_+.
\]
This yields
\begin{align*}
(2(\re(Y)-I)\cos \theta+2\im(Y)\sin\theta)s
&\geq -2(\re(Y) s\cos\theta_+ -\im(Y)s\sin\theta_+)\\
&=-2\re(\ov{\vep}_+Y)\geq -2s\|Y\|I.
\end{align*}
Next observe that, since $\re(Y)\leq \|\re(Y)\|I$, we have
\[
2\re(Y)-I\leq (2\|\re(Y)\|-1)I.
\]
Hence
\begin{align*}
&I+(2(\re(Y)-I)\cos \theta+2\im(Y)\sin\theta)s-(2\re(Y)-I)s^2\\
&\qquad\qquad \geq (1-2\|Y\|s-(2\|\re(Y)\|-1)s^2)I.
\end{align*}
Thus it suffices to show that the polynomial
\[
p(s)=1-2\|Y\|s-(2\|\re(Y)\|-1)s^2
\]
is positive on an open interval containing $0$. Since $p(0)=1>0$ it is clear that such an open neighborhood exists.

{\bf (POiii) $\Leftrightarrow$ (POiv):} The implication (POiii) $\Rightarrow$ (POiv)
is obvious. Conversely, assuming (POiv) holds, let $T_i=(1-\vep_i)B+\vep_i A=B+\vep_i(A-B)$ for
$|\vep_i|=r$, $i=1,2$. If $T_1$ and $T_2$ are contractions, than so is
$\half (T_1+T_2)=B-\half(\vep_1+\vep_2)(A-B)$. Now (POiii) follows because
\[
\left\{\frac{\vep_1+\vep_2}{2}\colon |\vep_1|=|\vep_2|=r\right\}
=\frac{r}{2}(\BT+\BT)=r\ov{\BD}=\{z\colon |z|\leq r\}.
\]
In particular, for $r$ in (POiii) we can take the same $r$ as in (POiv).

{\bf (POiii) $\Rightarrow$ (POii):}
By Lemma \ref{L:geo} the inequality \eqref{poscon} holds for each $\vep\in\BC$
with $|\vep|\leq r$.
It suffices to consider \eqref{poscon} for $\vep=-r,ir,-ir$. Taking $\vep=-r$ gives
\begin{align*}
0
&\leq (1+2r) D_B^2-2r\, \re(I-A^*B)-r^2(A-B)^*(A-B)\\
&\leq(1+r)^2 D_B^2-2r(1+r)\, \re(I-A^*B).
\end{align*}
The last inequality uses $(A-B)^*(A-B)\geq 2\re (I-A^*B)-D_B^2$, which
results from \eqref{re(I-A^*B)identity}.
Again by \eqref{re(I-A^*B)identity}, we see that $\re(I-A^*B)\geq0$.
Hence
\begin{equation}\label{RealBounds}
0\leq \re(I-A^*B)\leq \frac{(1+r)}{2 r} D_B^2.
\end{equation}
A similar argument, now with $\vep=ir$ and $\vep=-ir$, gives the
inequalities
\begin{equation}\label{ImagineBounds}
-\frac{1}{2r}D_B^2\leq\im(I-A^*B)\leq \frac{1}{2r}D_B^2.
\end{equation}
Hence by Lemma \ref{L:ReImIneq} we see that (POii) holds.

{\bf (POii) $\Rightarrow$ (POi):}
Applying Douglas' Lemma \cite{D66} to the identity in \eqref{ReYpos},
and using part (ii), we obtain that there exists an isometry
$\sbm{\wtilN\\\wtilM}$ mapping $\cD_B$ into $\sbm{\cD_A\\\cH_2}$
given by
\[
\mat{c}{\wtilN\\\wtilM}(2\re(Y)-I)^{\half}D_B=\mat{c}{D_A\\A-B}.
\]
Note that \eqref{ReYpos} implies $\cD_A\subset\cD_B$.
Set
\[
N=\wtilN (2\re(Y)-I)^{\half}\ands M=\wtilM (2\re(Y)-I)^{\half},
\]
so that
\begin{equation}\label{NMprops}
N D_B= D_A,\quad M D_B= A-B \ands N^*N+M^*M= 2\re (Y)-I.
\end{equation}
Now define $X=D_{B^*}M+B(I-Y^*)\in\cL(\cD_B,\cD_{B^*})$. Then
\begin{align*}
A-B
&=D_{B^*}^2 A-B(I- B^*A)
=D_{B^*}^2(A-B)+ D_{B^*}^2 B-B(I-A^*B)^*\\
&=D_{B^*}^2M D_B+D_{B^*}B D_B-BD_{B}Y^* D_B
=D_{B^*}^2(M+B(I-Y^*))D_B\\
&=D_{B^*}X D_B.
\end{align*}
Hence (POi) holds with this choice of $X$.

We have now proved the implications
\[
\mbox{(POi) } \Rightarrow \mbox{ (POii) } \Rightarrow \mbox{ (POiii) } \Rightarrow \mbox{ (POiv) } \Rightarrow \mbox{(POiii) } \Rightarrow \mbox{ (POii) } \Rightarrow \mbox{ (POi)},
\]
and hence the equivalence of (POi)--(POiv). The inefficiency in proving this equivalence is required for the computation of the inequalities of statement (iii).\medskip

\noindent {\bf Part III.}
Now we show that the equivalent definitions (POi)--(POiv) of $A\prec B$ define a pre-order, completing the proof of statement (ii). Clearly, $A\prec A$ for any contraction $A\in\cL_1(\cH_1,\cH_2)$. Assume $A\prec B$ and $B \prec C$. Let $X,\wtilX\in\cL(\cH_1,\cH_2)$ such that
\[
A-B=D_{B^*}X D_B\ands B-C=D_{C^*}\wtilX D_C.
\]
{}From (POiii) it follows that also $B^* \prec C^*$.
Following the argumentation in the proof of
(POii) $\Rightarrow$ (POi), we see that
there exist operators $N\in \cL(\cH_1)$
and $N_*\in\cL(\cH_2)$ such that
$D_B=N D_C$ and $D_{B^*}=N_* D_{C^*}$.
Hence $D_{B^*}=D_{B^*}^*=D_{C^*}^*N_*^*
=D_{C^*}N_*^*$. This implies
\[
A-C=A-B+B-C=D_{C^*}N_*^* X N D_C+D_{C^*}\wtilX D_C.
\]
Hence $A\prec C$. Thus $\prec$ is indeed a pre-order.\medskip

\noindent {\bf Part IV.} In the final part we prove statement (iii). The bounds for $\|Y\|$ in \eqref{XvsY} and \eqref{XnYineq} follow directly from the proofs of (POi) $\Rightarrow$ (POii) and (POiii) $\Rightarrow$ (POii). Indeed, applying Lemma \ref{L:ReImIneq} to \eqref{RealBounds} and \eqref{ImagineBounds}, yields $\|\re(Y)\|\leq  \frac{1+r}{2r}$ and $\|\im(Y)\|\leq\frac{1}{2r}$. Hence
\[
\|Y\|\leq \|\re(Y)\|+\|\im(Y)\|\leq  \frac{1+r}{2r}+\frac{1}{2r}
=\frac{2+r}{2r}.
\]
The identity $Y=I-X^*B$ in the proof of (POi) $\Rightarrow$ (POii) leads to the inequality for $\|Y\|$ in \eqref{XvsY}, via
\[
\|Y\|=\|I-X^*B\|\leq \|I\|+\|X^*\|\, \|B\|\leq 1+\|X\|.
\]
{}From the proof of (POii) $\Rightarrow$ (POi) we see that, given
$Y$ as in (POii), we may choose $X=D_B^* M+B(I-Y^*)$, with $M$
satisfying \eqref{NMprops}. Rewrite $X$ as
\begin{align*}
X&=D_B^* M+\half B(I-2\re(Y^*))+\half B(I-2i\im(Y^*))\\
&=D_B^* M+\half B(I-2\re(Y))+\half B(I+2i\im(Y)).
\end{align*}
Note that
\begin{align*}
\|Mx\|^2
&=\inn{M^*M x}{x}\leq \inn{(2\re(Y)-I)x}{x}=2\inn{\re(Y)x}{x}-\|x\|^2\\
&=2\|\re(Y)^\half x\|^2-\|x\|^2.
\end{align*}
Taking supremum over all $x\in\cD_B$ with $\|x\|=1$ we find
\begin{equation}\label{Mbound}
\|M\|^2\leq 2\|\re(Y)^\half\|^2-1=2\|\re(Y)\|-1.
\end{equation}
Note that
\[
\|I-2\re(Y)\|=2\|\re(Y)\|-1\ands
\|I-2i\im(Y^*)\|\leq 1+2\|\im (Y)\|.
\]
Hence
\begin{align*}
\|X\|
&\leq \|D_B^* M\|+\half \|B(I-2\re(Y))\|+\half \|B(I+2i\im(Y))\|\\
&\leq \|M\|+\half \|I-2\re(Y)\|+\half \|I+2i\im(Y)\|\\
&\leq \sqrt{2\|\re(Y)\|-1}+\|\re(Y)\|-\half+ \half+\|\im (Y)\|\\
&\leq\sqrt{2\|\re(Y)\|-1}+\|Y\|.
\end{align*}
Thus the bound on $\|X\|$ in \eqref{XvsY} applies. Combining this bound with the bound on $\|Y\|$ in \eqref{XnYineq} and the bound on $\|\re(Y)\|$ derived above, yields
\begin{align*}
\|X\|&\leq \sqrt{2\|\re(Y)\|-1}+\|Y\|\\
&\leq \sqrt{\frac{1+r}{r}-1}+\frac{2+r}{2r}=\frac{2+r}{2r}+r^{-1/2}
=\frac{2+2\sqrt{r}+r}{2r}.
\end{align*}

Finally, we derive the lower bound for $r$ in \eqref{rineq}.
For this purpose, recall from the proof of (POii) $\Rightarrow$ (POiii)
that \eqref{poscon} holds for all $\vep$ with $s=|\vep|$ such that
$p(s)=1-2\|Y\|s-(2\|\re(Y)\|-1)s^2\geq 0$. Hence the smallest (and
only) positive root of $p$ can serve as a lower bound for $r$. Thus assume
$s>0$, $p(s)=0$. If $\|\re(Y)\|=\half$, then $p(s)=0$ gives $s=\frac{1}{2\|Y\|}$,
in line with \eqref{rineq}. Assume $\|\re(Y)\|\not=\half$, i.e., $2\|\re(Y)\|-1>0$.
Then the smallest positive root of $p$ is then given by
\begin{align*}
&\frac{2\|Y\|-\sqrt{4\|Y\|^2+4(2\|\re(Y)\|-1)}}{-2(2\|\re(Y)\|-1)}
=\frac{\sqrt{\|Y\|^2+(2\|\re(Y)\|-1)} -\|Y\|}{2\|\re(Y)\|-1}\\
&\qquad\qquad=\frac{\|Y\|^2+(2\|\re(Y)\|-1) -\|Y\|^2}
{(2\|\re(Y)\|-1)(\sqrt{\|Y\|^2+(2\|\re(Y)\|-1)} +\|Y\|)}\\
&\qquad\qquad=\frac{1}
{\|Y\|+\sqrt{\|Y\|^2+2\|\re(Y)\|-1}}.
\end{align*}
Hence \eqref{rineq} holds also in this case.
\end{proof}

A few observations can be made directly from the definitions (POi)--(POiv)
of the pre-order relation and the preceding proof.

\begin{corollary}\label{C:directobs}
Assume $A\prec B$. Then
\begin{itemize}
\item[(i)] $A^*\prec B^*$;

\item[(ii)] $DAC \prec DBC$ for all $C\in\cL_1(\cH_0,\cH_1)$,
$D\in\cL_1(\cH_2,\cH_3)$;

\item[(iii)] $D_A=N D_B$ for some operator $N\in\cL(\cD_B,\cD_A)$
with $N^*N\leq 2\re(Y)-I$ where $Y$ is as in (POii);

\item[(iv)] $\cD_A\subset\cD_B$ and $A|_{\cD_B^\perp}=B|_{\cD_B^\perp}$.
\end{itemize}
\end{corollary}

\begin{proof}[\bf Proof.]
Observations (i) and (ii) can be derived immediately from (POiv) in Theorem \ref{T:prec}; (iii) was proved in the proof of the implication (POii) $\Rightarrow$ (POi). Part (iv) can be read off from (POi).
\end{proof}

Denote the equivalence relation defined by $\prec$ by $\sim$. The characterizations of $\sim$ given in the next result again go back to  \cite{KSS91}, the inequalities for $\|\wtilX\|$ and $\|\wtilY\|$ are new.

%

\begin{theorem}\label{T:eqrel}
Let $A,B\in\cL_1(\cH_1,\cH_2)$. Then $A\sim B$ if and only if one of the following equivalent statements holds:
\begin{itemize}
\item[(ERi)]
$ A-B=D_{A^*}\wtilX D_B \mbox{ for some } \wtilX\in\cL(\cD_B,\cD_{A^*})$;

\item[(ERii)]
$I-A^*B=D_A \wtilY D_B \mbox{ for some }\wtilY\in\cL(\cD_B)=\cL(\cD_B,\cD_A)$.

\end{itemize}
Moreover, if $A\sim B$, and $A\prec B$ holds with $X$ as in (POi) and $Y$ as in (POii) and $B \prec A$ holds as in (POi) with $X$ replaced by $X'$ and as in (POii) with $Y$ replaced by $Y'$, then $\wtilX$ and $\wtilY$ can be chosen in such a way that
\[
\|\wtilX\|\leq \|X\|\sqrt{2\|X'\|+1}\ands
\|\wtilY\|\leq \|Y\|\sqrt{2\|\re(Y')\|-1}.
\]
Additional bounds on $\|\wtilX\|$ and $\|\wtilY\|$ are obtained by replacing the roles of $X$ and $X'$, respectively $Y$ and $Y'$. Conversely, if $\wtilX$ and $\wtilY$ are as in (ERi) and (ERii) and $\wtilX'\in\cL(\cD_A,\cD_{B^*})$ satisfies $B-A=D_{B^*}\wtilX D_A$, then the operators $X$, $X'$, $Y$ and $Y'$ satisfy
\begin{align*}
\|X\|\leq \|\wtilX'\|^2+\|\wtilX'\|\sqrt{1+\|\wtilX'\|^2},\quad \|Y\|\leq 2\|\wtilY\|^2,\\
\|X'\|\leq \|\wtilX\|^2+\|\wtilX\|\sqrt{1+\|\wtilX\|^2},\quad \|Y'\|\leq 2\|\wtilY\|^2.
\end{align*}

\end{theorem}

\begin{proof}[\bf Proof.]
The proof is split into three parts.

\noindent{\bf Part I.} In the first part we show that $A\sim B$ implies (ERi) and (ERii) and we derive the bounds on $\|\wtilX\|$ and $\|\wtilY\|$. Hence, assume $A\sim B$, that is $A \prec B$ and $B\prec A$. Then also $B^*\prec A^*$,
by Corollary \ref{C:directobs}. Again by Corollary \ref{C:directobs}, we obtain
that there exist operators $N'\in\cL(\cD_A)$ and $N'_*\in\cL(\cD_{A^*})$ such that
$D_B=N' D_A$ and $D_{B^*}=N'_* D_{A^*}$. Then, with $X$ and $Y$ such that
$A-B=D_{B^*}XD_B$ and $I-A^*B=D_B YD_B$, we have
\begin{align*}
A-B=D_{B^*}X D_B=D_{A^*}N'^*_* X D_A,\\
I-A^*B=D_BYD_B=D_A N'^* Y D_B.
\end{align*}
Hence (ERi) and (ERii) hold with $\wtilX= N'^*_* X$, respectively $\wtilY=N'^* Y$.

In order to derive the bounds on $\|\wtilX\|$ and $\|\wtilY\|$, recall that by part (iii) in Corollary \ref{C:directobs} we can choose $N'$ and $N'_*$ above in
such a way that
\[
N'^*N'\leq 2\re(Y')-I\ands N'^*_*N'_*\leq 2\re(Y'_*)-I,
\]
with $Y'_*\in\cL(\cD_{A^*})$ such that $I-AB^*=D_{A^*}Y_*' D_{A^*}$.
In the same way as for the bound on $\|M\|$ in \eqref{Mbound} one derives
$\|N'\|^2\leq 2\|\re(Y')\|-1$. Hence
$\|\wtilY\|\leq \|N'\|\|Y\|\leq\|Y\|\sqrt{2\|\re(Y')\|-1}$.

Since $B-A=D_{A^*}X'D_A$, the relation $B^*\prec A^*$ is established
through $B^*-A^*=D_{A} X'^* D_{A^*}$. Following the proof of the implication
(POi) $\Rightarrow$ (POii), we see that we may take $Y'_*=I-X' A^*$. With this choice
for $Y_*'$ we get
\[
2\re(Y_*')-I=2I-X'A^*-AX'^*-I=I-2\re(X'A^*).
\]
Hence for each $x\in \cD_{A^*}$ with $\|x\|=1$
\begin{align*}
\|N_*' x\|^2&\leq\inn{(2\re(Y'_*)-I)x}{x}=\inn{(I-2\re(X'A^*))x}{x}\\
&=\|x\|^2-2\re\inn{A^*x}{X'^*x}\leq 1+2\|A^*x\|\|X'^*x\|\leq 1+2\|X'\|.
\end{align*}
Thus we obtain that $\|\wtilX\|\leq \|N_*'^*\|\|X\|\leq \|X\|\sqrt{2\|X'\|+1}$.\medskip

\noindent{\bf Part II.} Next we prove that (ERii) implies $A\sim B$ and compute the upper bounds on $\|Y\|$ and $\|Y'\|$. Assume $I-A^*B=D_A \wtilY D_B$ for some $\wtilY\in\cL(\cD_A)$. Then by \eqref{re(I-A^*B)identity} we have $D_A^2\leq 2\re (D_A \wtilY D_B)$. So for each $x\in \cH_1$
\[
\|D_A x\|^2=\inn{D_A^2 x}{x}\leq 2\re\inn{\wtilY D_Bx}{D_A x}\leq 2\|\wtilY\|\|D_B x\|
\| D_Ax\|.
\]
Hence $\|D_A x\|\leq 2\|\wtilY\|\|D_B x\|$ for each $x\in \cH_1$. This implies $D_A=M D_B$ for some $M\in\cL(\cH_1)$ with $\|M\|\leq 2\|\wtilY\|$. Thus $I-A^*B=D_B M^* \wtilY D_B$, and we conclude $A\prec B$ via (POii) with $Y=M^* \wtilY$. With $Y$ constructed in this way we have $\|Y\|\leq\|M^*\|\, \|\wtilY\|\leq 2\|\wtilY\|^2$, as claimed. Interchanging the roles of $A$ and $B$, and noting that $I-B^*A=D_B \wtilY^* D_A$, we find $B\prec A$ holds via $I-A^*B=D_B Y' D_B$ with $\|Y'\|\leq 2\|\wtilY^*\|^2=2\|\wtilY\|^2$. Hence $A\sim B$ holds along with the bounds on $\|Y\|$ and $\|Y'\|$.\medskip

\noindent{\bf Part III.} Finally we prove that $A\sim B$ follows from (ERi) and compute the upper bounds on $\|X\|$ and $\|X'\|$. Thus assume $A-B=D_{A^*}\wtilX D_B$ for some $\wtilX\in\cL(\cD_B,\cD_{A^*})$. Then
\[
I-A^*B-D_A^2=A^*A-A^*B=A^*D_{A^*}\wtilX D_B=D_{A}A^*\wtilX D_B.
\]
Using \eqref{re(I-A^*B)identity} we see that $D_B^2\leq 2\re(I-A^*B) -D_A^2$. Hence
\[
D_B^2\leq 2\re (D_{A}A^*\wtilX D_B + D_A^2)- D_A^2=D_A^2+2\re (D_{A}A^*\wtilX D_B).
\]
For any $x\in\cH_1$ we then find that
\[
\|D_Bx\|^2\leq\|D_Ax\|^2+2\re\inn{A^*\wtilX D_Bx}{ D_A x}
\leq \|D_Ax\|^2+2\|A^*\wtilX\|\|D_Bx\|\|D_Ax\|.
\]
Assume $D_Ax\not=0$. Set $\la=\|D_Bx\|/\|D_Ax\|$. Then the above inequality
can be expressed in terms of $\la$ as $\la^2\leq 1+2\|A^*\wtilX\|\la$. This
inequality is only satisfied on a bounded closed interval of 0, in fact, one easily computes that $\la^2\leq 1+2\|A^*\wtilX\|\la$ is equivalent to
\[
\|A^*\wtilX\|-\sqrt{1+\|A^*\wtilX\|^2}\leq\la\leq \|A^*\wtilX\|+\sqrt{1+\|A^*\wtilX\|^2}.
\]
In particular,
\[
\la\leq \|A^*\wtilX\|+\sqrt{1+\|A^*\wtilX\|^2}\leq \al \quad\mbox{with}\quad \al=\|\wtilX\|+\sqrt{1+\|\wtilX\|^2}.
\]
Observe that $\la\leq \al$ implies $\|D_Bx\|\leq\al \|D_Ax\|$ for $D_Ax\not=0$. Moreover, $D_Ax=0$ implies $D_Bx=0$, by the above bound on $\|D_Bx\|$, hence $\|D_Bx\|\leq\al \|D_Ax\|$ holds for any $x\in\cH_1$. This implies $D_B=MD_A$ for some $M\in\cL(\cD_A,\cD_B)$ with $\|M\|\leq \al$. Setting $X'=-\wtilX M$, we find that
\begin{align*}
B-A&=-D_{A^*}\wtilX D_B=-D_{A^*}\wtilX M D_A=D_{A^*}X' D_A,\\
\|X'\|&\leq \|\wtilX\|\,\|M\|=\|\wtilX\|^2+\|\wtilX\|\sqrt{1+\|\wtilX\|^2},
\end{align*}
and we conclude that $B\prec A$ and the claimed bound on $\|X'\|$ holds. Reversing the roles of $A$ and $B$, with $\wtilX'\in\cL(\cD_A,\cD_{B^*})$ such that $B-A=D_{B^*}\wtilX' D_A$, we find $A\prec B$ as in (POi) with $\|X\|\leq \|\wtilX'\|^2+\|\wtilX'\|\sqrt{1+\|\wtilX'\|^2}$.
\end{proof}

%
%
%
%

Now, for any $A\in\cL_1(\cH_1,\cH_2)$ we define the equivalence class
and cones
\begin{align*}
&\qquad\qquad\qquad[A]=\{B\in\cL_1(\cH_1,\cH_2)\colon A\sim B\},\\
&\vee(A)
=\{B\in\cL_1(\cH_1,\cH_2)\colon A\prec B\},\quad
\wedge(A)
=\{B\in\cL_1(\cH_1,\cH_2)\colon B\prec A\}.
\end{align*}

The following results are collected from \cite{S78}; a proof is added for completeness.

\begin{lemma}\label{L:ConesClasses}
Let $A\in\cL_1(\cH_1,\cH_2)$ be an isometry or a co-isometry and
let $B\in\cL_1^\circ (\cH_1,\cH_2)$. Then
\begin{align*}
&\qquad\qquad\qquad[A]=\{A\},\quad
\wedge(A)=\{A\},\\
&[B]=\vee(B)=\cL_1^\circ (\cH_1,\cH_2),\quad
\wedge(B)=\cL_1(\cH_1,\cH_2).
\end{align*}
\end{lemma}

\begin{proof}[\bf Proof.]
Assume $C$ is a contraction such that $A\sim C$. In
particular, $C\prec A$. Then $C-A=D_{A^*}X D_A$ for some operator $X$.
Since $A$ is an isometry or a co-isometry, either $D_A=0$ or $D_{A^*}=0$.
Hence $C=A$. This shows $\wedge(A)=\{A\}$. Since $A\in [A]\subset\wedge(A)=\{A\}$,
also $[A]=\{A\}$.

Since $\|B\|<1$, the defect operators $D_B$ and $D_{B^*}$ are invertible on $\cH_1$, respectively $\cH_2$. Thus for any $C\in\cL_1(\cH_1,\cH_2)$ we have $C-B=D_{B^*}XD_B$ with $X= D_{B^*}^{-1}(C-B)D_{B}^{-1}$, and thus $C\prec B$. This proves $\wedge(B)=\cL_1(\cH_1,\cH_2)$.

Next assume $B\prec C$, say $B-C=D_{C^*}XD_C$ with $X\in\cL(\cD_C,\cD_{C^*})$.
Assume $\|C\|=1$. Then there exists a sequence $v_1,v_2,\ldots\in\cH_1$
with $\|v_n\|=1$ for all $n$ and $\|C v_n\|\to 1$. This implies $D_C v_n\to 0$. Thus we have
\begin{align*}
\|B\|&\geq\|B v_n\|=\|Cv_n+D_{C^*}XD_C v_n\|\\
&\geq |\|Cv_n\|-\|D_{C^*}XD_C v_n\||\to|1-0|=1,
\end{align*}
in contradiction with $\|B\|<1$. Thus $\|C\|<1$, that is $C\in\cL_1^\circ (\cH_1,\cH_2)$.
Since $B\prec C$ holds for any $C\in\cL_1^\circ (\cH_1,\cH_2)$, we see that
$\vee(B)=\cL_1^\circ (\cH_1,\cH_2)$. Finally, note that
$[B]=\vee(B) \cap \wedge(B)=\cL_1^\circ (\cH_1,\cH_2)$.
\end{proof}


\section{Specification to analytic Laurent operators}
\setcounter{equation}{0}\label{S:LaurentPreOrder}

Throughout the remainder of this paper $\cU$ and $\cY$ denote separable Hilbert spaces. Since for any $K\in L^\infty(\cU,\cY)$ we have $\|K\|_\infty=\|L_K\|=\|T_K\|$, the pre-order and equivalence relations from Section \ref{S:BallPreOrder} induce pre-order and equivalence relations on the closed unit balls of the Banach spaces $L^\infty(\cU,\cY)$ and $H^\infty(\cU,\cY)$, denoted for both Banach spaces by $\precI$, respectively $\simI$, by restricting the pre-order structure and equivalence relation on $\cL_1(L^2_\cU, L^2_\cY)$ (or $\cL_1(H^2_\cU, H^2_\cY)$). Note that, in view of (POiii) and the fact that $\|L_K\|=\|T_K\|$ for each $K\in L^\infty(\cU,\cY)$, it follows that $L_{K_1}\prec L_{K_2}$ holds if and only if $T_{K_1}\prec T_{K_2}$. In this section we describe $\precI$ and $\simI$ in terms of the elements of $L^\infty_1(\cU,\cY)$ and $S(\cU,\cY)$.

We start with the Banach space $L^\infty(\cU,\cY)$. In order to characterize $\precI$, we require some more notation. For $K\in L^\infty(\cU,\cY)$ we define $K^*\in L^\infty(\cY,\cU)$ to be the function given by $K^*(e^{it})=K(e^{it})^*$, e.a\ on $\BT$. In case $K\in L^\infty_1(\cU,\cY)$ we define $D_K\in L^\infty(\cU,\cU)$ by $D_K(e^{it})=(I-K^*(e^{it})K(e^{it}))^{\half}$, a.e.\ on $\BT$. Note that $L_K^*=L_{K^*}$ and $D_{L_K}=L_{D_K}$.

\begin{theorem}\label{T:precLinfty}
Let $F,G\in L^\infty_1(\cU,\cY)$. Then $F\precI G$ if and only if one of the following four equivalent statements holds:
\begin{itemize}
\item[(LIPOi)]
$F-G=D_{G^*}Q D_G$, a.e.\ on $\BT$, for some $Q\in L^\infty(\cU,\cY)$;

\item[(LIPOii)]
$I-F^*G=D_G R D_G$, a.e.\ on $\BT$, for some $R\in L^\infty(\cU,\cU)$;

\item[(LIPOiii)]
there exist $r>0$ such that $(1-\vep)G+\vep F\in L^\infty_1(\cU,\cY)$ for all $\vep\in\BC,\,|\vep|\leq r$;

\item[(LIPOiv)]
there exist $r>0$ such that $(1-\vep)G+\vep F\in L^\infty_1(\cU,\cY)$ for all $\vep\in\BC,\,|\vep|=r$.
\end{itemize}
Moreover, the functions $Q$ and $R$ and constant $r>0$ can be chosen in such a way that
\begin{align*}
&\|Q\|_\infty\leq \|R\|_\infty+\sqrt{2\|\re(R)\|_\infty-1},\quad \|R\|_\infty\leq 1+\|Q\|_\infty,\\
&\|Q\|_\infty\leq \frac{2+2\sqrt{r}+1}{2r}, \quad \|R\|_\infty\leq \frac{2+r}{2r},\\
&r\geq \left(\|R\|_\infty+\sqrt{\|R\|_\infty^2+ 2\|\re(R)\|_\infty-1}\right)^{-1}.
\end{align*}
\end{theorem}

Before proving Theorem \ref{T:precLinfty} we derive a general result which will also be of use in Section \ref{S:Redheffer}.

\begin{proposition}\label{P:CLTapp}
Let $M\in L^\infty(\cW,\cX)$, $N\in L^\infty(\cU,\cV)$ and $P\in L^\infty(\cU,\cX)$. Assume there exists an operator $X\in \cL(L^2(\cV),L^2(\cW))$ such that $L_P=L_M X L_N$. Then without loss of generality $X=L_Q$ for some $Q\in L^\infty(\cV,\cW)$ with $\|Q\|_{\infty}=\|X\|$ and $P=MQN$.
\end{proposition}

\begin{proof}[\bf Proof.]
Set $\cH=\oran L_N$ and $\cH'=\oran L_M^*=\oran L_{M^*}$. Without loss of generality we may view $X$ as an operator from $\cH$ into $\cH'$. For any Hilbert space $\cK$ we write $V_\cK$ for the forward shift on $L^2(\cK)$. Then define $T=P_{\cH}V_\cV|_{\cH}$ and $T'=P_{\cH'}V_\cW|_{\cH'}$. Since $V_\cV L_N=L_N V_\cV$ and $V_\cW L_{M^*}=L_{M^*} V_\cW$ we have $V_\cV \cH\subset\cH$ and $V_\cW \cH'\subset\cH'$. This shows that $V_\cV$ and $V_\cW$ are unitary dilations of $T$ and $T'$, respectively, in the sense of \cite{NF70}. Moreover, since
\begin{align*}
L_M V_\cW X L_N=V_{\cX}L_M X L_N=V_{\cX} L_P=L_P V_{\cU}=L_M X L_N V_\cU
=L_M X V_\cV L_N,
\end{align*}
we obtain from the definitions of $\cH$ and $\cH'$, the fact that $V_\cV$ and $V_\cW$ are unitary dilations of $T$ and $T'$, respectively, and that $X$ maps $\cH$ into $\cH'$ that $T'X=XT$. Then by the Sz.-Nagy-Foias commutant lifting theorem \cite{NF70} there exists an operator $\wtilX\in \cL(H^2(\cV), H^2(\cW))$ with $\|\wtilX\|=\|X\|$, $P_{\cH'}\wtilX|_\cH=X$ and $V_\cW \wtilX=\wtilX V_\cV$. The latter says that $\wtilX=L_Q$ for some $Q\in L^\infty(\cV,\cW)$, with $\|Q\|_\infty=\|L_Q\|=\|\wtilX\|=\|X\|$. Since $P_{\cH'}L_Q|_\cH=X$, we have $L_P=L_M L_Q L_N=L_{MQN}$, that is $P=MQN$.
\end{proof}

\begin{proof}[\bf Proof of Theorem \ref{T:precLinfty}.]
We show that (LIPOi)--(LIPOiv) are equivalent to the corresponding statements (POi)--(POiv) for $L_F\prec L_G$. For (LIPOiii) and (PIPOiv) it suffices to observe that for any $\vep \in\BC$ we have
\[
L_{(1-\vep)G+\vep F}=(1-\vep)L_G+\vep L_F.
\]
Hence $(1-\vep)L_G+\vep L_F\in\cL_1(L^2_\cU, L^2_\cY)$ holds if and only if $(1-\vep)G+\vep F\in L^\infty_1(\cU,\cY)$. Moreover, (LIPOi), respectively (LIPOii), imply that (POi), respectively (POii), hold with $A=L_F$, $B=L_G$ and $X=L_Q$, respectively $Y=L_R$. It remains to show that if $L_F\prec L_G$, then one can achieve this via (POi) and (POii) with $X$ and $Y$ Laurent operators. However, this is a direct consequence of Proposition \ref{P:CLTapp} with $P=F-G$, $M=D_{G^*}$ and $N=D_G$ in case of (POi), and $P=I-F^*G$ and $M=N=D_G$ in case of (POii).

The bounds on $\|Q\|_\infty$, $\|R\|_\infty$ and $r$ follow directly from the bounds on $X$, $Y$ and $r$ in \eqref{XvsY}--\eqref{rineq} and the fact that $\|Q\|_\infty=\|L_Q\|$ and $\|R\|_\infty=\|L_R\|$.
\end{proof}

We proceed with the Banach space $H^\infty(\cU,\cY)$. Since we can view $H^\infty(\cU,\cY)$ as a sub-Banach space of $L^\infty(\cU,\cY)$, the pre-order $\precI$ on $H^\infty(\cU,\cY)$ is also characterized by (LIPOi)--(LIPOiv). However, for the purpose of the present paper we also need a characterization in terms of the values on $\BD$. Note that (LIPOiii) and (LIPOiv) simply extend to $\BD$, since $F,G\in H^\infty(\cU,\cY)$ implies $(1-\vep)G+\vep F \in H^\infty(\cU,\cY)$, and hence, if $F\precI G$ and $r$ is as in (LIPOiii) and (LIPOiv), then $(1-\vep)G+\vep F\in\ S(\cU,\cY)$, that is, $(1-\vep)G(\la)+\vep F(\la)\in\cL_1(\cU,\cY)$ for each  $\la\in\BD$ and $|\vep|\leq r$. This shows in particular that $F(\la)\prec G(\la)$ for each $\la\in\BD$. However, that $F\prec G$ holds pointwise on $\BD$ is in general not strong enough to conclude $F\precI G$.

\begin{corollary}\label{C:imps}
Let $F,G\in S(\cU,\cY)$. Then
\[
F\precI G\quad \Longrightarrow\quad F(\la) \prec G(\la)\ (\la\in\BD)\quad
\Longleftrightarrow \quad F(0) \prec G(0).
\]
Moreover, the implication $F(\la)\prec G(\la)$, $\la\in\BD$ $\Rightarrow$ $F\precI G$ in general does not hold.
\end{corollary}

\begin{proof}[\bf Proof.]
We already observed that the first implication holds. The second implication is a consequence of Theorem 2 in \cite{S76}. Indeed, applying Theorem 2 in \cite{S76} to a function $K\in S(\cU,\cY)$ one finds that for any $\la,\zeta\in\BD$ there exists an operator $T_{\la,\zeta}$ such that $K(\la)=K(\zeta)+D_{K(\zeta)^*}T_{\la,\zeta} D_{K(\zeta)}$, simply take $\la_1=\la_2=\la_3=\zeta$ in \cite[Theorem 2]{S76}. In other words, $K(\la)\prec K(\zeta)$ for any $\la,\zeta\in\BD$. Hence $K(\la)\sim K(\zeta)$ for any $\la,\zeta\in\BD$. Therefore, it suffices to check that $F(0) \prec G(0)$ in order to conclude that $F(\la) \prec G(\la)$ holds for all $\la\in\BD$.

Now take $\cU=\cY=\BC$, $F(\la)=0$ and $G(\la)=\la$, $\la\in\BD$. Then $F(0)=0=G(0)$. Hence $F(0)\sim G(0)$, and thus $F(\la)\sim G(\la)$ for each $\la\in\BD$. However, $L_F$ is a strict contraction and $L_G$ unitary. Hence, by Lemma \ref{L:ConesClasses} we have $L_F\not\sim L_G$, to be precise $L_G\prec L_F$ and $L_F\not\prec L_G$.
\end{proof}

Next we give another example that disproves the implication $F(\la)\prec G(\la)$, $\la\in\BD$ $\Rightarrow$ $F\precI G$.

\begin{example}\label{E:SCex1}
For $\om\in\BD$ and $\al\in[0,2\pi)$, let $\vph_{\om,\al}$ be the automorphism of the open unit disc given by $\vph_{\om,\al}(\la)=e^{-\al}\frac{\om-\la}{1-\ov{\om}\la}$, $\la\in\BD$. Then for each $\la\in\BD$, $|\vph_{\om,\al}(\la)|<1$. Hence, by Lemma \ref{L:ConesClasses}, for all $\om_1,\om_2\in\BD$ and $\al_1,\al_2\in[0,2\pi)$ we have $\vph_{\om_1,\al_1}(\la)\sim \vph_{\om_2,\al_2}(\la)$ for each $\la\in\BD$. Note that $\vph_{\om_1,\al_1}$ and $\vph_{\om_2,\al_2}$ are rational inner functions. Thus $L_{\vph_{\om_1,\al_1}}$ and $L_{\vph_{\om_2,\al_2}}$ are unitary. Then, by Lemma \ref{L:ConesClasses}, it follows that $L_{\vph_{\om_1,\al_1}}\prec L_{\vph_{\om_2,\al_2}}$, or equivalently $\vph_{\om_1,\al_1}\precI \vph_{\om_2,\al_2}$, holds if and only if $\vph_{\om_1,\al_1}=\vph_{\om_2,\al_2}$, i.e., $\om_1=\om_2$ and $\al_1=\al_2$.
\end{example}

In order to reverse the first implication of Corollary \ref{C:imps}, one needs to add a uniformity condition. Indeed, note that $(1-\vep)G+\vep F\in S(\cU,\cY)$, $|\vep|\leq r$, does not only imply $F(\la)\prec G(\la)$ as in (POiv) for any $\la\in\BD$, but also that we can achieve this with the same $r$ for each $\la$. In fact, one easily sees that this is equivalent to $F\precI G$. In Theorem \ref{T:precSC} below we will derive similar variations on (POi)--(POiii) that are equivalent to $F\precI G$. First we extend the functions $K^*$, $D_K$ and $D_{K*}$ to $\BD$ by setting $K^*(\la)=K(\la)^*$, $D_K(\la)=(I-K^*(\la)K(\la))^\half$ and $D_{K^*}(\la)=(I-K(\la)K^*(\la))^\half$ for each $\la\in\BD$.

\begin{lemma}\label{L:eqincl}
Let $K\in S(\cU,\cY)$. Then for all $\la,\zeta\in\BD$ and a.e.\ $e^{it}\in\BT$
\begin{align*}
&\ran D_K(e^{it})\subset \ran D_K(\la)=\ran D_K(\zeta),\\
&\ran D_{K^*}(e^{it})\subset \ran D_{K^*}(\la)=\ran D_{K^*}(\zeta).
\end{align*}
\end{lemma}

\begin{proof}[\bf Proof.]
The equality of $\ran D_K(\la)$ and $\ran D_K(\zeta)$ follows directly from Theorem 1 of \cite{S76}. Then inclusion $\ran D_{K^*}(e^{it})\subset \ran D_{K^*}(\la)$ is a direct consequence of the maximum principle.
\end{proof}

In view of Lemma \ref{L:eqincl} we set $\cD_K:=\oran D_K(\la)$ and $\cD_K^*:=\oran D_{K^*}(\la)$, where $\la\in\BD$ is chosen arbitrarily. We then have $\oran D_K(e^{it})\subset \cD_K$ and $\oran D_{K^*}(e^{it})\subset \cD_K^*$ for a.e.\ $e^{it}\in\BT$.

\begin{theorem}\label{T:precSC}
Let $F,G\in S(\cU,\cY)$. Then $F\precI G$ if and only if one of the following four equivalent conditions holds:
\begin{itemize}
\item[(SCPOi)]
$F-G=D_{G^*}Q D_{G}$ for some $\cL(\cD_G,\cD_G^*)$-valued bounded function $Q$ on $\BD$;

\item[(SCPOii)]
$I-F^*G=D_{G}R D_{G}$ for some $\cL(\cD_G)$-valued bounded function $R$ on $\BD$;

\item[(SCPOiii)]
there exist $r>0$ such that $(1-\vep)G+\vep F\in\bS(\cU,\cY)$ for all $\vep\in\BC,\,|\vep|\leq r$;

\item[(SCPOiv)]
there exist $r>0$ such that $(1-\vep)G+\vep F\in\bS(\cU,\cY)$ for all $\vep\in\BC,\,|\vep|= r$.
\end{itemize}
Moreover, if $F\precI G$, then the functions $Q$ and $R$ and constant $r>0$ can be chosen in such a way that
\begin{equation}\label{SCbounds}
\begin{aligned}
&\|Q\|_\infty\leq \|R\|_\infty+\sqrt{2\|\re(R)\|_\infty-1},\quad \|R\|_\infty\leq 1+\|Q\|_\infty,\\
&\|Q\|_\infty\leq \frac{2+2\sqrt{r}+1}{2r}, \quad \|R\|_\infty\leq \frac{2+r}{2r},\\
&r\geq \left(\|R\|_\infty+\sqrt{\|R\|_\infty^2+ 2\|\re(R)\|_\infty-1}\right)^{-1}.
\end{aligned}
\end{equation}
\end{theorem}

\begin{proof}[\bf Proof.]
That statements (SCPOiii) and (SCPOiv) are mutually equivalent, and equivalent to $F\precI G$ is obvious from Theorem \ref{T:precLinfty}. Hence, to complete the proof, it suffices to prove that (SCPOi), (SCPOii) and (SCPOiii) are equivalent, along with the bounds on $\|Q\|_\infty$, $\|R\|_\infty$ and $r$.

Assume (SCPOiii) holds. Then $F(\la)\prec G(\la)$ holds for each $\la\in\BD$. Therefore, by Theorem \ref{T:prec}, that there exist functions $Q$ and $R$ on $\BD$, with values in $\cL(\cD_G,\cD_{G}^*)$, respectively $\cL(\cD_G)$, such that $F-G= D_{G^*}Q D_{G}$ and $I-F^*G=D_{G}R D_{G}$, simply let the value of $Q$ and $R$ at $\la\in\BD$ be the operators $X$ and $Y$ obtained from (POi) and (POii), respectively, that establish $F(\la)\prec G(\la)$. It remains to show that $Q$ and $R$ are bounded. However, since for each $\la\in\BD$ in (POiii) we can take the same value for $r$, it follows that $\|Q(\la)\|\leq\frac{2+2\sqrt{r}+1}{2r}$ and $\|R(\la)\|\leq \frac{2+r}{2r}$ for each $\la\in\BD$. Hence $Q$ and $R$ are bounded, and the bounds on $\|Q\|_\infty$ and $\|R\|_\infty$ apply.

Assume (SCPOi) holds. Again via the observation that $F(\la)\prec Q(\la)$, $\la\in\BD$, one defines a $\cL(\cD_G)$-valued function $R$ on $\BD$ such that $I-F^*G=D_{G}R D_{G}$, and it remains to show that $R$ is bounded on $\BD$. This, however, follows directly from the second bound of part (iii) of Theorem \ref{T:prec}: For each $\la\in\BD$
\[
\|R(\la)\|\leq 1+\|Q(\la)\|\leq 1+\|Q\|_\infty,\quad \mbox{hence}\quad \|R\|_\infty\leq 1+\|Q\|_\infty.
\]
In a similar fashion one proves that (SCPOii) implies (SCPOi), by deriving the bound on $\|Q\|_\infty$ in terms of $\|R\|_\infty$.

Finally, assume (SCPOii) holds. Then, $F(\la)\prec G(\la)$, for each $\la\in\BD$, with $Y=R(\la)$. Hence, by part (ii) of Theorem \ref{T:prec}, $2\re( R(\la))\geq I$, which shows that $2\|\re(R)\|_\infty-1 \geq 2\|\re{R(\la)}\|-1\geq0$ and $\|R\|_\infty\geq \|\re(R)\|_\infty\geq \half$. Hence there exists an $r_\la>0$ such that $\|(1-\vep)G(\la)+\vep F(\la)\|\leq 0$ for each $|\vep|\leq r_\la$ and
\begin{align*}
r_\la
&\geq \frac{1}{\|R(\la)\|+\sqrt{\|R(\la)\|^2+2\|\re(R(\la))\|-1}}\\
&\geq  \frac{1}{\|R\|_\infty+\sqrt{\|R\|_\infty^2+2\|\re(R)\|_\infty-1}}.
\end{align*}
Since the right hand side in the last inequality is positive and independent of $\la$, we can take this as the value for $r$ in (SCPOiii) and (SCPOiv) and it is immediate that  the lower bound on $r$ in \eqref{SCbounds} holds.
\end{proof}

Since $F\precI G$ is equivalent to $L_F\prec L_G$, the next results follow immediately from Corollary \ref{C:directobs} and Lemma \ref{L:ConesClasses}.

\begin{corollary}\label{C:directobs2}
\begin{itemize}
\item[(i)]
If $F,G\in L^\infty_1(\cU,\cY)$, $F\precI G$ and $H\in L^\infty_1(\cY,\cZ)$, $K\in L^\infty_1(\cX,\cU)$, then $HFK\precI HGK$;

\item[(ii)]
the sets $L^\infty_{<1}(\cU,\cY)$ and $S_0(\cU,\cY)$ form equivalence classes with respect to the pre-order $\precI$; if $F\in L^\infty_1(\cU,\cY)$, resp.\ $F\in S(\cU,\cY)$, and $G\in  L^\infty_{<1}(\cU,\cY)$, resp.\ $G\in  S_0(\cU,\cY)$, then $F\precI G$.

\item[(iii)]
If $G\in S(\cU,\cY)$ is inner, then the set $\{G\}$ forms an equivalence class
and $F\precI G$ implies $F=G$ for any $F\in S(\cU,\cY)$.

\end{itemize}
\end{corollary}

Part (iii) of Corollary \ref{C:directobs}, in the context of the pre-order $\precI$ on $S(\cU,\cY)$ can be rephrased in terms of the associated de Branges-Rovnyak spaces. Recall that for a given Schur class function $K\in S(\cU,\cY)$, the associated de Branges-Rovnyak space, denoted by $\cH(K)$, is equal to $\ran D_{T_K^*}\subset H^2_\cU$ with the lifted norm $\|D_{T_K^*} h\|_{\cH(K)}=\| h\|_{H^2_\cU}$, $h\in\cD_{T_K^*}$.

\begin{corollary}\label{C:dBRconnect}
Let $F,G\in S(\cU,\cY)$ such that $F\precI G$. Then $\cH(F)\subset \cH(G)$, as subsets of $H^2_\cU$, and there exists an isometry mapping $\cH(F)$ into $\cH(G)$.
\end{corollary}

\begin{proof}[\bf Proof.]
Since $F\precI G$ we have $T_F\prec T_G$ in $\cL_1(H^2_\cU,H^2_\cY)$. By part (i) of Corollary \ref{C:directobs} we have $T_F^*\prec T_G^*$, and thus, by part (iii) of the same corollary, there exists a bounded operator $X$ such that $X D_{T_G^*}=D_{T_F^*}$, i.e., $D_{T_G^*}X^*=D_{T_F^*}$. The latter identity shows $\ran D_{T_F^*}\subset \ran D_{T_G^*}$, i.e., $\cH(F)\subset \cH(G)$, as subsets of $H^2_\cU$. Moreover, set $\wtil{\cH}_0=D_{T_G^*}(\cD_{T_F^*})$. Then $\wtil{\cH}_0$ is a (not necessarily closed) linear manifold in $\cH(G)$. Write $\wtil{\cH}$ for the closure of $\wtil{\cH}_0$ in $\cH(G)$. For any $k=D_{T_G^*}h\in\wtil{\cH}_0$, $h\in\cD_{T_F^*}$, we have
\[
\|X k\|_{\cH(G)}=\|XD_{T_G^*}h\|_{\cH(F)}=\|D_{T_F^*}h\|_{\cH(F)}=\|h\|_{H^2_\cU} =\|D_{T_G^*}h\|_{\cH(G)}=\|k\|_{\cH(G)}.
\]
It follows that $X$ maps $\wtil{\cH}$ isometrically onto $\cH(F)$, and hence $P_{\wtil{\cH}}X^*$ maps $\cH(F)$ isometrically into $\cH(G)$, where adjoint is taken viewing $X$ as a map from $\cH(G)$ to $\cH(F)$.
\end{proof}

In case $G\in S(\cU,\cY)$ is not inner but does attain norm one at some point of $\ov{\BD}$, then $F\precI G$ still implies $F$ shows similar behavior at this point in the direction where $G$ attains norm one. This is a consequence of Proposition \ref{P:limbehavior} below. We first prove the following lemma.

\begin{lemma}\label{L:limbehavior}
Let $F,G\in S(\cU,\cY)$ such that $F\precI G$ and let $Q$ be as in (SCPOi). Then for any $u\in\cU$ and $\la\in\BD$ we have
\[
\|G(\la)u-F(\la)u\| \leq\|Q\|_\infty \sqrt{2\|u\|(\|u\|-\|G(\la)u\|)}.
\]
\end{lemma}

\begin{proof}[\bf Proof.]
First note that
\begin{align*}
\|D_{G}(\la)u\|^2
&=\|u\|^2-\|G(\la)u\|^2=(\|u\|-\|G(\la)u\|)(\|u\|+\|G(\la)u\|)\\
&\leq 2\|u\|(\|u\|-\|G(\la)u\|).
\end{align*}
Hence
\begin{align*}
\|G(\la)u-F(\la)u\|
&=\|D_{G^*}(\la)Q(\la)D_{G}(\la)u\|\leq  \|D_{G^*}(\la)Q(\la)\|\, \|D_{G}(\la)u\|\\
&\leq \|Q(\la)\|\,\sqrt{\|D_{G}(\la)u\|^2}
\leq\|Q\|_\infty \sqrt{2\|u\|(\|u\|-\|G(\la)u\|)},
\end{align*}
and our claim follows.
\end{proof}

\begin{proposition}\label{P:limbehavior}
Let $F,G\in S(\cU,\cY)$ such that $F\precI G$ and let $u\in\cU$ and $y\in\cY$ with $\|u\|=\|y\|$. Let $t\mapsto \la_t$, $t\in(0,1]$ be a continuous curve in $\ov{\BD}$ with $\la_t\in\BD$ whenever $t\in(0,1)$. Then the following statements holds:
\begin{itemize}
\item[(i)]
If $\lim_{t\uparrow 1} G(\la_t)u=y$, then $\lim_{t\uparrow 1} F(\la_t)u=y$.

\item[(ii)]
If $\lim_{t\uparrow 1}\|G(\la_t)u\|=\|u\|$, then $\lim_{t\uparrow 1}\|F(\la_t)u\|=\|u\|$.

\item[(iii)]
If $\lim_{t\uparrow 1}\|G(\la_t)\|=1$, then $\lim_{t\uparrow 1}\|F(\la_t)\|=1$.
\end{itemize}
In particular, if $\beta\in\BT$ and $\lim_{\la\to\beta}G(\la)x=y$ nontangentially (respectively unrestrictedly), then $\lim_{\la\to\beta}F(\la)x=y$ nontangentially (respectively unrestrictedly).
\end{proposition}

\begin{proof}[\bf Proof.]
Claim (i) follows directly from Lemma \ref{L:limbehavior}. In order to see that (ii) holds, note that for each $\la\in\BD$
\begin{align*}
\|u\|-\|F(\la)u\|
&=\|u\|-\|G(\la)u\|+\|G(\la)u\|-\|F(\la)u\|\\
&\leq \|u\|-\|G(\la)u\|+|\,\|G(\la)u\|-\|F(\la)u\|\,|\\
&\leq\|u\|-\|G(\la)u\|+||G(\la)u-F(\la)u\|\\
&\leq \|u\|-\|G(\la)u\|+ \|Q\|_\infty \sqrt{2\|u\|(\|u\|-\|G(\la)u\|)}.
\end{align*}
Hence, replacing $\la$ with $\la_t$, we see that $\lim_{t\uparrow 1}\|G(\la_t)u\|=\|u\|$ implies that $\lim_{t\uparrow 1}\|F(\la_t)u\|=\|u\|$.

Finally, assume $\lim_{t\uparrow 1}\|G(\la_t)\|=1$. Let $\vep>0$. Define $\de=\|Q\|_\infty+\vep-\|Q\|_\infty\sqrt{\|Q\|_\infty+2\vep}$, and note that $\de=\half(\|Q\|_\infty-\sqrt{\|Q\|_\infty^2+2\vep})^2$ is the positive solution of $\de+\|Q\|_\infty\sqrt{2\de}=\vep$. Since $\lim_{t\uparrow 1}\|G(\la_t)\|=1$ and $\de>0$, there exists a $t_\de\in(0,1)$ such that $1-\|G(\la_t)\|=|1-\|G(\la_t)\|\,|<\de$ for all $t\in(t_\de,1)$. Hence for each $t\in(t_\de,1)$ there exists a $u_t\in\cU$, $\|u_t\|=1$ such that $\|u_t\|-\|G(\la_t)u_t\|<\de$, and thus
\begin{align*}
|1-\|F(\la_t)u_t\|\,|
&=\|u_t\|-\|F(\la_t)u_t\|\\
&\leq \|u_t\|-\|G(\la_t)u_t\|+ \|Q\|_\infty \sqrt{2\|u_t\|(\|u\|-\|G(\la_t)u_t\|)}\\
&<\de+\|Q\|_\infty\sqrt{2\de}=\vep.
\end{align*}
Therefore $1\geq \|F(\la_t)\|>1-\vep$ for each $t\in(t_\de,1)$. Hence $\lim_{t\uparrow 1}\|F(\la_t)\|=1$.
\end{proof}

Applying Theorem \ref{T:eqrel} along with arguments similar to those in the proofs of Theorems \ref{T:precLinfty} and \ref{T:precSC} one obtains the following characterizations of $\simI$. Details are omitted.

\begin{theorem}\label{T:eqrelSC}
Let $F,G\in L^\infty_1(\cU,\cY)$. Then $F\simI G$ if and only if one of the following equivalent statement holds:
\begin{itemize}
\item[(LIERi)] $ F-G=D_{F^*}\wtilQ D_G
\mbox{ a.e.\ on $\BT$ for some } \wtilQ\in L^\infty(\cU,\cY)$;

\item[(LIERii)] $I-F^*G=D_F \wtilR D_G
\mbox{ a.e.\ on $\BT$ for some }\wtilR\in L^\infty(\cU,\cU)$.

\end{itemize}
Moreover, if $F\simI G$ holds with $Q$ and $R$ the $L^\infty$-functions from (LIPOi), respectively (LIPOii), relating to $F\precI G$ and $Q'$ and $R'$ the $L^\infty$-functions from (LIPOi), respectively (LIPOii), relating to $G\precI F$, then $\wtilQ$ and $\wtilR$ can be chosen in such a way that
\begin{equation}\label{tilQtilRbounds}
\|\wtilQ\|_\infty\leq \|Q\|_\infty\sqrt{2\|Q'\|_\infty+1}\ \ \mbox{and}\ \
\|\wtilR\|_\infty\leq \|R\|_\infty\sqrt{2\|\re(R')\|_\infty-1}.
\end{equation}
Additional bounds on $\|\wtilQ\|$ and $\|\wtilR\|$ are obtained by replacing the roles of $Q$ and $Q'$, and $R$ and $R'$. Furthermore, in case $F,G\in S(\cU,\cY)$, then $F\simI G$ is equivalent to (LIERi) and (LIERii) with $\wtilQ$, respectively $\wtilR$, bounded functions on $\BD$ and a.e.\ on $\BT$  and the bounds on $\|\wtilQ\|_\infty$ and $\|\wtilR\|_\infty$ in \eqref{tilQtilRbounds} remain to hold.
\end{theorem}

In the remainder of this section we focus on properties of the functions $Q$ and $R$ and $\wtilQ$ and $\wtilR$ in Theorems \ref{T:precSC} and \ref{T:eqrelSC}, respectively. That these functions can be chosen to be $L^\infty$-functions on $\BT$ was already observed in Theorems \ref{T:precLinfty} and \ref{T:precSC}. We now consider their behavior on $\BD$, just for the case that $\cU$ and $\cY$ are finite dimensional.

\begin{proposition}\label{P:QRbehavior}
Let $F,G\in S(\cU,\cY)$ with $\cU$ and $\cY$ finite dimensional. Assume $F\precI G$ (respectively $F\simI G$). Then the functions $Q$ and $R$ in (SCPOi) and (SCPOii) (resp.\ $\wtilQ$ and $\wtilR$ in (LIERi) and (LIERii)) can be chosen to be continuous on $\BD$. Moreover, assume the nontangential limits of both $F$ and $G$ exist in $e^{it}\in\BT$. Then $R(e^{it})$ can be defined in line with (LIPOii) and such that $\lim_{z\to e^{it}}P_{\cD_{G(e^{it})}}R(z)P_{\cD_{G(e^{it})}}=R(e^{it})$ as $z\in\BD$ converges to $e^{it}$ nontangentially.
\end{proposition}

In order to prove Proposition \ref{P:QRbehavior} we first prove two preparational results, which can be viewed as extensions of Douglas' lemma. We do not require finite dimensionality at this stage.

\begin{proposition}\label{P:ExtendedDouglas}
Let $X_t\in\cL(\cH,\cH_1)$, $Y_t\in\cL(\cH,\cH_2)$, $t\in(0,1]$. Assume
\begin{itemize}
\item[(i)] $X_t^*X_t\geq Y_t^*Y_t$ for all $t\in(0,1]$;

\item[(ii)] $\lim_{t\uparrow 1} X_t=X_1$ and $\lim_{t\uparrow 1}Y_t=Y_1$, both in the strong operator topology.
\end{itemize}
Then there exist contractions $Z_t\in\cL_1(\cH_1,\cH_2)$ such that $Y_t=Z_tX_t$ for each $t\in(0,1]$ and $\lim_{t\uparrow 1} Z_t w= Z_1 w$ for any $w\in\oran X_1$.
\end{proposition}

\begin{proof}[\bf Proof.]
By Douglas' lemma \cite{D66}, condition (i) implies that, for each $t\in(0,1]$, there exists a unique contraction $Z_t\in\cL_1(\cH_1,\cH_2)$ with $Y_t=Z_tX_t$ and $\kr Z_t=\kr X_t^*$.

Next we prove that $\lim_{t\uparrow 1} Z_t=Z_1$ in the strong operator topology on $\oran X_1$. Let $w\in\oran X_1$, $\|w\|=1$. Fix $\vep>0$. Determine a $v\in \cH_1$ such that $\|w-X_1v\|< \vep/4$. Now take $t_0\in(0,1)$ such that
\begin{equation}\label{XYSOTcon}
\|Y_1v-Y_tv\|<\vep/4,\quad \|X_1v-X_tv\|<\vep/4\quad (t\in(t_0,1]).
\end{equation}
Note that \eqref{XYSOTcon} implies that for $t\in(t_0,1]$
\begin{align*}
\|(Z_1-Z_t)X_1v\|
&=\|Z_1X_1v-Z_tX_tv+Z_tX_tv-Z_tX_1v\|\\
&\leq\|Z_1X_1v-Z_tX_tv\|+\|Z_tX_tv-Z_tX_1v\|\\
&\leq\|Y_1v-Y_tv\|+\|X_tv-X_1v\|<\vep/2.
\end{align*}
Therefore, for all $t\in(t_0,1]$ we have
\begin{align*}
\|Z_1w-Z_tw\|
&=\|Z_1w-Z_1X_1v+Z_1X_1v-Z_tX_1v+Z_tX_1v-Z_tw\|\\
&\leq \|Z_1(w-X_1v)\|+\|(Z_1X_1-Z_tX_1)v\|+\|Z_t(X_1v-w)\|\\
&\leq 2\|w-X_1v\|+\|(Z_1X_1-Z_tX_1)v\|<\vep.
\end{align*}
We conclude that $Z_tw\to Z_1w$ for each $w\in \oran X_1$.
\end{proof}

\begin{corollary}\label{C:QRbehavior}
Let $U$ and $V$ be $\cL(\cH,\cH_1)$- and $\cL(\cH,\cH_2)$-valued functions on $\BD$, continuous in the strong operator topology, such that $U(\la)^*U(\la)\leq V(\la)^*V(\la)$ for each $\la\in\BD$. Then there exists an $\cL(\cH_1,\cH_2)$-valued function $W$ on $\BD$ such that $\|W\|_\infty\leq 1$ and $U(\la)=W(\la)V(\la)$, $\la\in\BD$, and for each $\zeta\in\BD$ and $x\in\oran V(\zeta)$ the function $\la\mapsto W(\la)x$ is continuous at $\zeta$. Furthermore, in case the nontangential limits of $U$ and $V$ exist in $e^{it}\in\BT$, then $W$ can be defined in $e^{it}$, maintaining the identity $U=WV$ and $\|W\|_{\infty}\leq 1$ and such that $\lim_{\la\to e^{it}}W(\la)x=W(e^{it})x$ nontangentially for each $x\in\oran V(e^{it})$.
\end{corollary}

\begin{proof}[\bf Proof.]
Applying Douglas' lemma for each $\la\in\BD$, we can define $W$ pointwise, also in the point of $\BT$ where $U$ and $V$ both have nontangential limits. To see that $\la\mapsto W(\la)x$ is continuous at $\zeta\in\BD$ for each $x\in \oran V(\zeta)$, apply Proposition \ref{P:QRbehavior} with $X_t=U(\la_t)$ and $Y_t=V(\la_t)$ for any continuous curve $t\mapsto\la_t$, $t\in(0,1]$ in $\BD$ with $t_1=\zeta$. Similarly, if the nontangential limits of $U$ and $V$ exist at $e^{it}\in\BT$ and $x\in\oran V(e^{it})$, then one finds that the nontangential limit of $\la\mapsto W(\la)x$ exist at $e^{it}$ by considering the values of $U$ and $W$ along continuous curves in $\BD$ that converge to $e^{it}$ nontangentially.
\end{proof}

We are now almost ready to prove Proposition \ref{P:QRbehavior}. Let $K\in S(\cU,\cY)$. Note that
$K$ and $K^*$ converge nontangentially to $K\in L^\infty(\cU,\cY)$ and $K^*\in L^\infty(\cY,\cU)$ a.e.\ on $\BT$ in the strong operator topology, cf., \cite[Section V.2]{NF70}. Assuming $\cU$ and $\cY$ to be finite dimensional, the convergence occurs in any norm, hence also in the operator norm, and it follows that $D_{K}$ and $D_{K^*}$ converge nontangentially to $D_K\in L^\infty(\cU,\cU)$ and $D_{K^*}\in L^\infty(\cY,\cY)$ a.e.\ on $\BT$.

\begin{proof}[\bf Proof of Proposition \ref{P:QRbehavior}.]
In order to prove Proposition \ref{P:QRbehavior} we have to recall the pointwise construction of $Q$, $R$, $\wtilQ$ and $\wtilR$ from the proofs of Section \ref{S:BallPreOrder}. Assume $F\precI G$.

We start with the construction of $R$. Following the proof of the implication (POiii) $\Rightarrow$ (POii) of Theorem \ref{T:prec} along with the proofs of Lemmas \ref{L:ReImIneq} and \ref{L:Basic}, for each $\la\in\BD$ we set
\begin{center}
$T_1(\la)=(\re(I-F(\la)^*G(\la)))^{\half},$\\[.2cm]
$T_2(\la)=(\im(I-F(\la)^*G(\la))_+)^{\half}\ands
T_3(\la)=(\im(I-F(\la)^*G(\la))_-)^{\half}$.
\end{center}
Here $\im(I-F(\la)^*G(\la))_+$ and $\im(I-F(\la)^*G(\la))_-$ are positive semidefinite operators on $\cU$ such that
\[
\im(I-F(\la)^*G(\la))=\im(I-F(\la)^*G(\la))_+-\im(I-F(\la)^*G(\la))_-.
\]
Define $\eta_1=\frac{1+r}{2r}$ and $\eta_2=\eta_3=\frac{1}{2r}$, with $r$ as in (POiii).
Then
\[
T_k(\la)^*T_k(\la)\leq \eta_k D_G(\la)^2=\eta_k D_G(\la)^* D_G(\la),\mbox{ for }k=1,2,3.
\]
By taking the spectral projections in the definitions of $\im(I-F(\la)^*G(\la))_+$ and $\im(I-F(\la)^*G(\la))_-$ in the right way, we can arrange the functions $\la\mapsto T_k(\la)$ to be continuous in $\BD$ for $k=2,3$. Clearly, $\la\mapsto T_1(\la)$ defines a continuous function on $\BD$. Applying Corollary \ref{C:QRbehavior} for $k=1,2,3$ with $U=T_k$ and $V=\sqrt{\eta_k}D_G$, where we note that $\oran V(\la)=\cD_G$ for each $\la\in\BD$, we obtain that there exists a continuous $\cL(\cU)$-valued function $L_k$ on $\BD$, setting $L_k|_{\cU\ominus\cD_G}=0$, such that $\|L_k\|_\infty\leq \sqrt{\eta_k}$ and $T_k=L_k D_G$. Then $R=L_1^* L_1+i (L_2^*L_2-L_3^*L_3)$ satisfies $I-F^*G=D_G R D_G$ and $\|R\|_{\infty}\leq \sum_{k=1}^3\|L_1\|_\infty^2\leq \sum_{k=1}^3\eta_k<\infty$.

For each $k=1,2,3$ the function $L_k$ is continuous on $\BD$. Since $\cU$ and $\cY$ are finite dimensional, also $L_k^*$ is continuous on $\BD$. This implies $R$ is continuous on $\BD$. The statement about the nontangential limits of $R$ follows by applying the statement from Corollary \ref{C:QRbehavior} concerning the nontangential limits in the construction of $R$, as above. Note here that Corollary \ref{C:QRbehavior} only gives the nontangential convergence of $L_i|_{\cD_{G(e^{it})}}$, and the inclusion $\cD_{G(e^{it})}\subset \cD_G$ can be strict. Thus we obtain that $L_i$ can be defined in $e^{it}$, setting $L_i(e^{it})|_{\cD_{G(e^{it})}^\perp}=0$, such that
\begin{align*}
\lim_{z\to e^{it}}P_{\cD_{G(e^{it})}}L_i(z)^*L_i(z)P_{\cD_{G(e^{it})}} &=P_{\cD_{G(e^{it})}}L_i(e^{it})^*L_i(e^{it})P_{\cD_{G(e^{it})}}\\
&=L_i(e^{it})^*L_i(e^{it}),
\end{align*}
as $z$ approaches $e^{it}$ nontangentially. The claim about the nontangential limits of $R$ now follows since $R=L_1^* L_1+i (L_2^*L_2-L_3^*L_3)$.

Next we consider the construction of $Q$ from $R$, following the proof of the implication (POii) $\Rightarrow$ (POi) of Theorem \ref{T:prec}. Hence, let $R$ be as constructed in the first part of the proof. By part (ii) of Theorem \ref{T:prec}, with $A=F(\la)$, $B=G(\la)$ and $Y=R(\la)$ for some $\la\in\BD$, we have $\kr (2\re(R(\la))-I)=\{0\}$, i.e., $\oran(2\re(R(\la))-I)=\cD_G$, and this identity holds independently of the choice of $\la$. Formula \eqref{ReYpos} yields
\[
D_G(2\re(R)-I)D_G=D_F^2+(F-G)^*(F-G)\geq (F-G)^*(F-G).
\]
Now take $U=F-G$ and $V=(2\re(R)-I)^{\half}D_G$. Then $U$ and $V$ are continuous and the nontangential limits of $U$ and $V$ in $e^{it}\in\BT$ exists in case the nontangential limits of $F$ and $G$ exist in $e^{it}$. Moreover, $\oran V(\la)=\cD_G$ for each $\la\in\BD$. Hence, by Corollary \ref{C:QRbehavior} we obtain that there exists a $\cL(\cD_G,\cY)$-valued continuous function $\wtilM$ on $\BD$ such that $\|\wtilM\|_{\infty}\leq 1$ and $\wtilM V=U$. Then the function $Q=D_{G^*}\wtilM (2\re(R)-I)^{\half}+G(I-R^*)$ satisfies $F-G=D_{G^*}QD_G$ and clearly is continuous on $\BD$.

We proceed with the statement concerning $\wtilQ$ and $\wtilR$. Thus assume $F\simI G$. Following the constructions from the proof of Theorem \ref{T:eqrel}, with $A=F(\la)$ and $B=G(\la)$, with $\la\in\BD$ arbitrary, and $R$ and $Q$ as defined above, again using Formula \eqref{ReYpos} we obtain
\[
D_G(2\re(R)-I)D_G\geq D_F^2.
\]
This shows, using Corollary \ref{C:QRbehavior} with $U=D_F$ and $V=(2\re(R)-I)^{\half}D_G$, that there exists a $\cL(\cD_G)$-valued continuous function $\wtilN'$ with $\|\wtilN'\|_\infty\leq 1$ and $\wtilN'(2\re(R)-I)^{\half}D_G=\wtilN'V=U=D_F$. Then $N'=\wtilN'(2\re(R)-I)^{\half}$ is a continuous and bounded function on $\BD$ with $N'D_G=D_F$ and $\wtil{R}=N'^*R$ is continuous and satisfied the claims of Theorem \ref{T:eqrelSC}. Similarly, now using that $L_F\prec L_G$ implies $L_F^* \prec L_G^*$, one obtains that there exists a $\cL(\cD_G^*)$-valued continuous bounded function $\wtilN'_*$ on $\BD$ such that $N'_*D_{G^*}=D_{F^*}$, in which case $\wtilQ=N'^*_*Q$ is continuous and satisfied the corresponding claims of Theorem \ref{T:eqrelSC}.
\end{proof}

The limitations of Proposition \ref{P:QRbehavior}, finite dimensionality and limited information about the boundary behavior, cannot easily be remedied by extensions of Proposition \ref{P:ExtendedDouglas} and Corollary \ref{C:QRbehavior}. The following two examples illustrate the complications. We start with the case where $\cU$ and $\cY$ are infinite dimensional. In this case the complication in the proof of Proposition \ref{P:QRbehavior} is that not only the functions $L_i$ need to be continuous, but $L_i^*L_i$ as well.

\begin{example}
Set $\vph(\la)=\sin^2(\pi/|\la|)$, $\la\in\BD\backslash\{0\}$ and $\vph(0)=0$. Let $S$ denote the forward shift operator on $\ell^2_+$ and $\lfloor a\rfloor$ the smallest integer greater than $a\in\BR$. Define $U(\la)=Z(\la)+I_{\ell^2_+}$ with $Z(\la)=\vph(\la)S^{*\lfloor 1/|\la|\rfloor}$ and $V(\la)=2I_{\ell^2_+}$ for each $\la\in\BD$. Although $\vph$ is not continuous at $0$, $Z$ is continuous on $\BD$ in the strong operator topology. However, $Z^*$ is not continuous at $\la=0$, since $\|Z^*(\la)u\|=\|\vph(\la)S^{\lfloor 1/|\la|\rfloor}u\|=\vph(\la)\|u\|$. Thus $U$ is continuous on $\BD$, but $U^*$ is not. Clearly $V$ is continuous on $\BD$ and we have $U^*U\leq V^*V$. The function $W$ with $U=WV$ must be equal to $U$, and, as remarked above, is indeed continuous in the strong operator topology. Identifying $\BC^k$ with the first $k$ entries of $\ell^2_+$, we obtain that
\[
W^*(\la)W(\la)=I+\vph(\la)^2(I_{\ell^2_+}-P_{\BC^{\lfloor 1/|\la|\rfloor}})+Z(\la)+Z^*(\la).
\]
The first summand is constant, the two in the middle converge to 0 in the strong operator topology as $|\la|\to 0$, but the last one does. Consequently, $W^*W$ is not continuous in the strong operator topology.
\end{example}

As a result of Lemma \ref{L:eqincl} and Corollary \ref{C:QRbehavior}, we can prove that $Q$, $R$, $\wtilQ$ and $\wtilR$ are continuous on $\BD$, and not just in specific directions. However, on $\BT$ the inclusion $\cD_{G(e^{it})}\subset \cD_G$ may be strict, and via Corollary \ref{C:QRbehavior} we only achieve convergence in directions from $\cD_{G(e^{it})}$. The next example shows it may indeed not be possible, in the context of Corollary \ref{C:QRbehavior}, to have convergence in the other directions.

\begin{example}
Define
\[
V(\la)=\mat{cc}{1&(1-|\la|^2)^\half\\ (1-|\la|^2)^\half& 1-|\la|^2}=\mat{cc}{1&\al_\la^\half \\ \al_\la^\half & \al_\la}\quad (\la\in\BD).
\]
Here $\al_\la=1-|\la|^2$. Note that $V(\la)$ is positive definite on $\BD$ and positive semi-definite on $\BT$. The singular value decomposition of $V$ on $\BD$ is given by $V=Z^*DZ$ with
\begin{align*}
Z(\la)&=\mat{cc}{-2\al_\la^\half\rho_{+,\la}^{-\half}&-2\al_\la^\half\rho_{-,\la}^{-\half}\\
((1-2\al_\la)-\sqrt{1+4\al_\la^2})\rho_{+,\la}^{-\half} &(1-2\al_\la)+\sqrt{1+4\al_\la^2})\rho_{-,\la}^{-\half}},\\
D(\la)&=\mat{cc}{\frac{1+2\al_\la+\sqrt{1+4\al_\la^2}}{2}&0\\0&\frac{1+2\al_\la-\sqrt{1+4\al_\la^2}}{2}}.
\end{align*}
Here $\rho_{\pm,\la}=(2-2\al_\la+8\al_\la^2)\mp 2(1-2\al_\la)\sqrt{1+4\al_\la^2}$. For $\la\in\BT$ we have $Z(\la)=I_2$ and $D(\la)$ is defined as above. Note that the entries depend on $\al_\la$ only, hence on $|\la|$. As $|\la|\uparrow 1$, we have $\rho_{+,\la}\to 0$ and $\rho_{-,\la}\to 4$. From this we immediately obtain that the second column of $Z(\la)$ converges to $\sbm{0\\1}$ as $|\la|\uparrow 1$. Since $Z(\la)$ is unitary for each $\la\in\ov{\BD}$, the two columns are perpendicular unit vectors, and it follows that the first column of $Z(\la)$ converges to $\sbm{1\\0}$ as $|\la|\uparrow 1$. We conclude that $Z$ is continuous on $\ov{\BD}$, and so is its inverse $Z^*$.

Now define $\vph(\la)=\sin^2(\frac{1}{1-\re(\la)})$, $\la\in\ov{\BD}\backslash\{1\}$, and $\vph(1)=0$. Then $\vph$ is continuous on $\ov{\BD}\backslash\{1\}$, but not at $1$, and has values in $[0,1]$. Set $U(\la)=Z(\la)^*\wtilD(\la)Z(\la)$ for each $\la\in\ov{\BD}$, where $D(\la)=\diag(1,\vph(\la))D(\la)$. Write $d_1(\la)$ and $d_2(\la)$ for the diagonal elements of $D(\la)$. Since $d_2(\la)\to 0$ as $|\la|\uparrow 1$, $\wtilD$ is continuous on $\ov{\BD}$. Moreover, on $\ov{\BD}$ we have
\[
V(\la)^*V(\la)-U(\la)^*U(\la)=Z(\la)^*\diag(d_1(\la)^2, (1-\vph(\la)^2)d_2(\la)^2)Z(\la)\geq0.
\]
The function $W$ with $WV=Z$ is given by $W(\la)=Z(\la)^*\diag(1,\vph(\la)) Z(\la)$, is continuous on $\ov{\BD}\backslash\{1\}$, but not at $1$, since otherwise we would have that $\diag(1,\vph(\la))=Z(w)W(\la)Z(\la)^*$  is continuous at $1$ as well.
\end{example}

%

\section{Invariance of $\precI$ and $\simI$ under Redheffer maps}
\setcounter{equation}{0}\label{S:Redheffer}

In this section we investigate the extend to which the pre-order $\precI$ and equivalence relation $\simI$ are invariant under Redheffer maps. In particular, we prove Theorems \ref{T:RedInvar} and \ref{T:RedInvar2}. Let
\begin{equation}\label{Phi}
\Phi=\mat{cc}{\Phi_{11}&\Phi_{12}\\\Phi_{22}&\Phi_{22}}\in S(\cE'\oplus\cU,\cE\oplus\cY),
\end{equation}
with $\cU$, $\cY$, $\cE$ and $\cE'$ separable Hilbert spaces. Recall that the Redheffer map associated with $\Phi$ is defined by
\begin{equation}\label{Redheffer}
\begin{aligned}
\fR_\Phi[K]&=\Phi_{22}+\Phi_{21}K(I-\Phi_{11}K)^{-1}\Phi_{12}\\
&\qquad \mbox{with }K\in S(\cE,\cE'),\ \|\Phi_{11}(0)K(0)\|<1.
\end{aligned}
\end{equation}
Here all operations are pointwise.

Before proving Theorem \ref{T:RedInvar} we start with some preparations. The first are some well-known inequalities, cf., \cite[Page 732]{S78} and \cite[Section XIV.1]{FF90}.

\begin{lemma}\label{L:EssIneq}
Let $\Phi$ be as in \eqref{Phi} and $K\in S(\cE,\cE')$ with $\|\Phi_{11}(0)K(0)\|<1$. Then the following inequalities hold in each point of $\BD$:
\begin{equation}\label{Ineqs}
\begin{aligned}
D_{\fR_\Phi[K]}^2&\geq \Phi_{12}^*(I-\Phi_{11}K)^{-*}D_K^2(I-\Phi_{11}K)^{-1}\Phi_{12};\\
D_{\fR_\Phi[K]^*}^2&\geq \Phi_{21}(I-\Phi_{11}K)^{-1}D_{K^*}^2(I-\Phi_{11}K)^{-*}\Phi_{21}^*.
\end{aligned}
\end{equation}
Moreover, assume the nontangential limits of $K$, $\Phi$ and $(I-\Phi_{11}K)^{-1}$ in $e^{it}\in\BT$ exist, then the inequalities \eqref{Ineqs} extends to $e^{it}$, with identity in the first inequality in case $\Phi$ is inner and identity in the second inequality in case $\Phi$ is $*$-inner.
\end{lemma}

Applying Douglas' Lemma \cite{D66} pointwise to the the inequalities \eqref{Ineqs}, provides the following corollary.

\begin{corollary}\label{C:confuncts}
Let $\Phi$ be as in \eqref{Phi} and $K\in S(\cE,\cE')$ with $\|\Phi_{11}(0)K(0)\|<1$. Then there exists an $\cL_1(\cD_{\fR_\Phi[K]},\cD_K)$-valued function $L$ and an $\cL_1(\cD_{\fR_\Phi[K]}^*,\cD_K^*)$-valued functions $L_*$, both defined on $\BD$ and in in a.e.\ $e^{it}\in\BT$ where the nontangential limits of $(I-\Phi_{11}K)^{-1}$ exist, such that
\[
L D_{\fR_\Phi[K]}=D_{K}(I-\Phi_{11}K)^{-1}\Phi_{12} \ands
D_{\fR_\Phi[K]^*}L_*=\Phi_{21}(I-K\Phi_{11})^{-1}D_{K^*}.
\]
Moreover, in case $\Phi$ is inner (resp.\ $*$-inner), then $L(e^{it})|_{\cD_{\fR_\Phi[K](e^{it})}}$ is an isometry (resp.\ $L_*(e^{it})^*|_{\cD_{\fR_\Phi[K]^*(e^{it})}}$ is an isometry) for a.e.\ $e^{it}\in\BT$ as above.
\end{corollary}

\begin{lemma}\label{L:EssComp}
Let $\Phi$ be as in \eqref{Phi} and $K_1,K_2\in S(\cE,\cE')$ with $\|\Phi_{11}(0)K_i(0)\|<1$ for $i=1,2$. Then
\begin{equation}\label{EssComp}
\fR_\Phi[K_1]-\fR_\Phi[K_2]=
\Phi_{21}(I-K_1\Phi_{11})^{-1}(K_1-K_2)(I-\Phi_{11}K_2)^{-1}\Phi_{12}.
\end{equation}
\end{lemma}

\begin{proof}[\bf Proof.]
The identity is a consequence of the following computation.
\begin{align*}
&\fR_\Phi[K_1]-\fR_\Phi[K_2]
=\Phi_{21}(K_1(I-\Phi_{11}K_1)^{-1}-K_2(I-\Phi_{11}K_2)^{-1})\Phi_{12}\\
&\qquad=\Phi_{21}((I-K_1\Phi_{11})^{-1}K_1-K_2(I-\Phi_{11}K_2)^{-1})\Phi_{12}\\
&\qquad=\Phi_{21}(I\!-\!K_1\Phi_{11})^{-1}(K_1(I\!-\!\Phi_{11}K_2)\!-\!(I\!-\!K_1\Phi_{11})K_2)(I\!-\!\Phi_{11}K_2)^{-1}\Phi_{12}\\
&\qquad=\Phi_{21}(I-K_1\Phi_{11})^{-1}(K_1-K_2)(I-\Phi_{11}K_2)^{-1}\Phi_{12}.\qedhere
\end{align*}
\end{proof}

\begin{proof}[\bf Proof of Theorem \ref{T:RedInvar}.]
By Corollary \ref{C:directobs2}, part (i), we have $\Phi_{11} F_1\simI \Phi_{11} F_2$. Next apply Proposition \ref{P:limbehavior}, part (iii), with $F=\Phi_{11} F_1$, $G= \Phi_{11} F_2$ and $\la_t=0,\ t\in(0,1]$. It follows that $\|\Phi_{11}(0)F_1(0)\|=1$ implies $\|\Phi_{11}(0)F_2(0)\|=1$. Reversing the roles of $F_1$ and $F_2$, we see that $\|\Phi_{11}(0)F_1(0)\|=1$ holds if and only if $\|\Phi_{11}(0)F_2(0)\|=1$. This proves the first claim, because $\|\Phi_{11}(0)F_i(0)\|\leq1$ for $i=1,2$.

Since $F_1\simI F_2$, there exists a bounded function $\wtilQ$ on $\BD$ such that $F_1-F_2=D_{F_1^*}\wtilQ D_{F_2}$. Hence, together with identity \eqref{EssComp} we can conclude that
\[
\fR_\Phi[F_1]-\fR_\Phi[F_2]=
\Phi_{21}(I-F_1\Phi_{11})^{-1}D_{F_1^*}\wtilQ D_{F_2}(I-\Phi_{11}F_2)^{-1}\Phi_{12}.
\]
Now let $L$ and $L_*$ be as in Corollary \ref{C:confuncts}, with $K$ replaced by $F_1$ in the definition of $L$ and with $K$ replaced by $F_2$ in the definition of $L_*$, i.e.,
\[
L D_{\fR_\Phi[F_2]}=D_{F_2}(I-\Phi_{11}F_2)^{-1}\Phi_{12} \ands
D_{\fR_\Phi[F_2]^*}L_*=\Phi_{21}(I-F_1\Phi_{11})^{-1}D_{F_1^*}.
\]
Then we obtain that
\[
\fR_\Phi[F_1]-\fR_\Phi[F_2]=
D_{\fR_\Phi[F_2]^*}L_*\wtilQ L D_{\fR_\Phi[F_2]},
\]
and $\|L_*\wtilQ L\|_\infty\leq \|\wtilQ\|_\infty<\infty$. Hence $\fR_\Phi[F_1]\simI\fR_\Phi[F_2]$.
\end{proof}

\begin{corollary}\label{C:StrictSchur}
Let $\Phi$ be as in \eqref{Phi}. Then $\fR_\Phi$ maps $S_0(\cE,\cE')$ into $S_0(\cU,\cY)$ if and only if $\Phi_{22}\in S_0(\cU,\cY)$.
\end{corollary}

\begin{proof}[\bf Proof.]
Since $\Phi_{22}=\fR_\Phi[\un{0}]\in S_0(\cU,\cY)$, with $\un{0}\in S_0(\cE,\cE')$ indicating the constant function with value $0\in\cL(\cE,\cE')$, the necessity of the condition $\Phi_{21}\in S_0(\cU,\cY)$ is evident. The converse follows directly from Theorem \ref{T:RedInvar} and Corollary \ref{C:directobs2}, part (ii).
\end{proof}

Corollary \ref{C:StrictSchur} also appears in \cite{AD08}, see Theorem 2.27, part (4), and Page 219, and is a special case of \cite[Theorem 6.1]{tH11} (provided the non-required assumption $\Phi_{11}(0)=0$ in \cite{tH11} is left out). Proposition \ref{P:specialcase2} below provides a refinement of Corollary \ref{C:StrictSchur} under an additional constraint on $\Phi$.

While Redheffer maps preserve the equivalence relation $\simI$, this is not necessarily the case for the pre-order $\precI$. An example considered by Bakonyi \cite{B95} proves this, as discussed in the introduction. We present another example where the coefficient function $\Phi$ is an operator polynomial of degree one.

\begin{example}\label{E:nonprecinv}
Let $\de_0,\de_1,\ldots\in(0,1)$ with $\lim_{k\to\infty}\de_k=1$. Set
\[
N=\diag_{k\in\BN}(\de_k)\ons \ell^2_+ \ands
M=\diag_{k\in\BN}(\rho_k)\ons \ell^2_+\mbox{ with }\rho_k=\sqrt{1-\de_k^2}.
\]
Note that $N^*=N$ and $M=D_N$. Now take
\[
\Phi(\la)=\mat{cc}{\la N&M\\-\la M&N}=\mat{cc}{\la N&D_{N^*}\\-\la D_N&N^*}\ons\mat{c}{\ell^2_+\\\ell^2_+}\quad (\la\in\BD).
\]
Then $\Phi$ is an inner function. Let $F_1,F_2\in S(\ell^2_+,\ell^2_+)$ be given by $F_1(\la)=I_{\ell^2_+}$ and $F_2(\la)=0$ for $\la\in\BD$. Then $F_1\precI F_2$, by Corollary \ref{C:directobs2}. We claim that $\fR_\Phi[F_1]\ \not\!\!\precI \fR_\Phi[F_2]$. Set $G_j=\fR_\Phi[F_j]$, $j=1,2$. Then
\[
G_1(\la)=N-\la M(I-\la N)^{-1}M,\quad G_2(\la)=N\quad (\la\in\BD).
\]
Writing out $M$ and $N$ yields
\[
\mbox{$G_1(\la)=\diag_{k\in\BN}\left(\frac{\de_k-\la}{1-\la\de_k}\right) =\diag_{k\in\BN}(\vph_{\de_k,1}(\la))$}\quad (\la\in\BD).
\]
Here $\vph_{\de_k,1}$ is the inner function defined in Example \ref{E:SCex1}. Moreover, $D_{G_2}(\la)=D_{G_2^*}(\la)=M$ for each $\la\in\BD$, and we have
\[
G_1(\la)-G_2(\la)=-\la M(I-\la N)^{-1}M.
\]
Since $|\de_i|<1$ for each $i$, we have $\kr M=\{0\}$. This implies that in the formula
$G_1-G_2=D_{G_2^*}Q D_{G_2}$, we necessarily have
\[
Q(\la)=-\la(I-\la N)^{-1}=\diag_{k\in\BN}\left(\frac{-\la}{1-\la\de_k}\right).
\]
Note that $\sup_{\la\in\BD}|\frac{-\la}{1-\la\de_k}|\geq\frac{\de_k}{1-\de_k^2}$. The fact that $\lim_{k\to\infty}\de_k=1$, thus implies that $\lim_{k\to\infty}\sup_{\la\in\BD}|\frac{-\la}{1-\la\de_k}|=\infty$. This shows that $Q$ is not bounded on $\BD$. Hence $G_1\not\!\!\precI G_2$, as claimed.
%
\end{example}

We conclude this paper with the special case of Theorem \ref{T:RedInvar2}, i.e., $\Phi_{11}\in S_0(\cE',\cE)$, where the Redheffer maps do preserve the pre-order $\precI$.

\begin{proof}[\bf Proof of Theorem \ref{T:RedInvar2}.]
Since $F\precI G$, there exists a bounded function $Q$ on $\BD$ such that $F-G=D_{G^*}QD_G$. Define $L$ and $L_*$ as in Corollary \ref{C:confuncts}, with $K$ replaced by $G$. Then by \eqref{EssComp}, with $K_1=F$ and $K_2=G$, we have
\begin{align*}
\fR_\Phi[F]-\fR_\Phi[G]
&=\Phi_{21}(I-F\Phi_{11})^{-1}(F-G)(I-\Phi_{11}G)^{-1}\Phi_{12}\\
&=\Phi_{21}(I-F\Phi_{11})^{-1}D_{G^*}QD_G(I-\Phi_{11}G)^{-1}\Phi_{12}\\
&=\Phi_{21}(I-F\Phi_{11})^{-1}D_{G^*}Q L D_{\fR_\Phi[G]}.
\end{align*}
Thus, we have $\fR_\Phi[F]\precI\fR_\Phi[G]$ if there exists a bounded $\cL(\cD_G^*,\cD_{\fR_\Phi[G]}^*)$-valued function $N$ on $\BD$ such that
\[
\Phi_{21}(I-F\Phi_{11})^{-1}D_{G^*}=D_{\fR_\Phi[G]^*}N.
\]
Note that
\begin{align*}
&\Phi_{21}(I-F\Phi_{11})^{-1}D_{G^*}
=\Phi_{21}(I-G\Phi_{11})^{-1}D_{G^*}+\\
&\qquad\qquad\qquad\qquad\qquad\qquad+\Phi_{21}((I-F\Phi_{11})^{-1}-(I-G\Phi_{11})^{-1})D_{G^*}\\
&\qquad\qquad=D_{\fR_\Phi[G]^*}L_*
+\Phi_{21}(I-G\Phi_{11})^{-1}(F-G)(I-F\Phi_{11})^{-1}D_{G^*}\\
&\qquad\qquad=D_{\fR_\Phi[G]^*}L_*
+\Phi_{21}(I-G\Phi_{11})^{-1}D_{G^*}QD_G(I-F\Phi_{11})^{-1}D_{G^*}\\
&\qquad\qquad=D_{\fR_\Phi[G]^*}L_*(I+QD_G(I-F\Phi_{11})^{-1}D_{G^*}).
\end{align*}
Hence we can take $N=L_*(I+QD_G(I-F\Phi_{11})^{-1}D_{G^*})$, which is bounded on $\BD$, since $I-F\Phi_{11}$ is by assumption boundedly invertible.
\end{proof}

With a few extra assumptions, in addition to $\Phi_{11}\in S_0(\cE',\cE)$, the behavior of $\fR_\Phi$ improves even further. We start with some simple observations based on the inequalities in \eqref{Ineqs}.

\begin{corollary}\label{C:specialcase1}
Let $\Phi$ be as in \eqref{Phi} with $\Phi_{11}\in S_0(\cE',\cE)$. If $\Phi$ is inner (resp.\ $*$-inner), then $\fR_\Phi$ maps inner (resp.\ $*$-inner) functions in $S(\cE,\cE')$ to inner (resp.\ $*$-inner) functions in $S(\cU,\cY)$.
\end{corollary}

\begin{proof}[\bf Proof.]
Assume $\Phi$ is inner. By Lemma \ref{L:EssIneq}, the first inequality in \eqref{Ineqs} extend to a.e.\ point of $\BT$, where it becomes an identity rather than an inequality. Now if $K\in S(\cE,\cE')$ is inner, then $D_K(e^{it})=0$ for a.e.\ $e^{it}\in\BT$. Thus a.e.\ on $\BT$ the right hand side in the first inequality in \eqref{Ineqs} is 0, and hence also $D_{\fR_\Phi[K]}=0$ a.e.\ on $\BT$. Thus $\fR_\Phi[K]$ is inner. The claim about the case where $\Phi$ is $*$-inner is proved similarly, now using the second inequality of \eqref{Ineqs}.
\end{proof}

Given $K\in H^\infty(\cU,\cY)$, we say $K$ has an $H^\infty$-inverse in case $K(\la)$ is invertible for each $\la\in\BD$ and $\la\mapsto K(\la)^{-1}$ in $H^\infty(\cY,\cU)$; note that this is equivalent to $T_K$ being invertible. The function $\la\mapsto K(\la)^{-1}$ will be denoted by $K^{-1}$.

\begin{proposition}\label{P:specialcase2}
Let $\Phi$ be as in \eqref{Phi} such that $\Phi_{12}$ or $\Phi_{21}$ has an $H^\infty$-inverse. Then $\Phi_{11}\in S_0(\cE',\cE)$, $\Phi_{22}\in S_0(\cU,\cY)$ and $\fR_\Phi$ maps $S_0(\cE',\cE)$ into $S_0(\cU,\cY)$. Moreover, in case $\Phi_{12}$ has an $H^\infty$-inverse, then for any $F\in S(\cE,\cE')$ we have
\[
\dim \cD_F \leq \dim \cD_{\fR_\Phi[F]} \ands  \dim(\cU\ominus \cD_{\fR_\Phi[F]})\leq \dim(\cE\ominus \cD_F),
\]
and for a.e.\ $e^{it}\in\BT$
\begin{align}
\dim \oran D_{F(e^{it})} \leq \dim \oran D_{\fR_\Phi[F](e^{it})},\notag\\
\dim(\cU\ominus \oran D_{\fR_\Phi[F](e^{it})})\leq \dim(\cE\ominus \oran D_{F(e^{it})}),\label{DimIneqs1}
\end{align}
with equality in \eqref{DimIneqs1} in case $\Phi$ is inner. Similarly, if $\Phi_{21}$ has an $H^\infty$-inverse, then for any $F\in S(\cE,\cE')$ we have
\[
\dim \cD_F^* \leq \dim \cD_{\fR_\Phi[F]}^* \ands  \dim(\cY\ominus \cD_{\fR_\Phi[F]}^*)\leq \dim(\cE'\ominus \cD_F^*),
\]
and for a.e.\ $e^{it}\in\BT$
\begin{align}
\dim \oran D_{F^*(e^{it})} \leq \dim \oran D_{\fR_\Phi[F]^*(e^{it})},\notag\\
\dim(\cU\ominus \oran D_{\fR_\Phi[F]^*(e^{it})})\leq \dim(\cE\ominus \oran D_{F^*(e^{it})}),\label{DimIneqs2}
\end{align}
with equality in \eqref{DimIneqs2} in case $\Phi$ is $*$-inner.
\end{proposition}

\begin{proof}[\bf Proof.]
Identifying $H^2(\cE'\oplus\cU)$ with $H^2(\cE')\oplus H^2(\cU)$ and $H^2(\cE\oplus\cY)$ with $H^2(\cE)\oplus H^2(\cY)$, we can write
\[
T_\Phi=\mat{cc}{T_{\Phi_{11}}& T_{\Phi_{12}}\\ T_{\Phi_{21}}& T_{\Phi_{22}}}.
\]
Now if $\Phi_{12}$ or $\Phi_{21}$ has an $H^\infty$-inverse, then $T_{\Phi_{12}}$ or $T_{\Phi_{21}}$ is invertible, and thus, since $T_\Phi$ is contractive, both $T_{\Phi_{11}}$ and $T_{\Phi_{22}}$ are strict contractions, that is, $\Phi_{11}\in S_0(\cE',\cE)$ and $\Phi_{22}\in S_0(\cU,\cY)$. In particular, $\fR_\Phi$ maps $S_0(\cE,\cE')$ into $S_0(\cU,\cY)$, by Corollary \ref{C:StrictSchur}.

Assume $\Phi_{12}$ has an $H^\infty$-inverse. Fix an arbitrary $\la\in\BD$. By the first inequality in \eqref{Ineqs}, $(I-\Phi_{11}(\la)F(\la))^{-1}\Phi_{12}(\la)$ maps  $\cU\ominus \cD_{\fR_\Phi[F]}=\kr D_{\fR_\Phi[F](\la)}$ into $\cE\ominus \cD_F=\kr D_{F(\la)}$. Since $\Phi_{11}\in S_0(\cE',\cE)$, the operator $(I-\Phi_{11}(\la)F(\la))^{-1}\Phi_{12}(\la)$ is invertible, and thus $\dim(\cU\ominus \cD_{\fR_\Phi[F]})\leq \dim(\cE\ominus \cD_F)$. Furthermore, the fact that $(I-\Phi_{11}(\la)F(\la))^{-1}\Phi_{12}(\la)$ is invertible, implies
\[
\oran D_{F(\la)}^2(I-\Phi_{11}(\la)F(\la))^{-1}\Phi_{12}(\la)=\oran D_{F(\la)}=\cD_F.
\]
Again using the first inequality in \eqref{Ineqs} we see that $\Phi_{12}(\la)^*(I-\Phi_{11}(\la)F(\la))^{-*}$ is invertible and maps $\cD_F$ into $\oran D_{\fR_\Phi[F](\la)}=\cD_{\fR_\Phi[F]}$. Hence $\dim \cD_F \leq \dim \cD_{\fR_\Phi[F]}$. The proof of the inequalities \eqref{DimIneqs1} at a point $e^{it}\in\BT$ is similar. In case $\Phi$ is inner, the first inequality in \eqref{Ineqs} extends to an identity a.e.\ on $\BT$, and thus
\[
(I-\Phi_{11}F)\Phi_{12}^{-*}D_{\fR_\Phi[F]}^2\Phi_{12}^{-1}(I-\Phi_{11}F)=D_F^2\quad \mbox{a.e.\ on } \BT.
\]
A similar argument as used in deriving the inequalities \eqref{DimIneqs1} now shows the inequalities \eqref{DimIneqs1} also hold in the opposite direction. Hence we have \eqref{DimIneqs1} a.e.\ with equalities rather than inequalities. The proof for the case that $\Phi_{21}$ has an $H^\infty$-inverse follows a similar path, now using the second inequality in \eqref{Ineqs}.
\end{proof}

Finally, we consider the question when $\fR_\Phi$ induces an injective map between the equivalence classes. In general this is not the case, simply take take $\Phi_{12}=\un{0}$, $\Phi_{21}=\un{0}$ and $\Phi_{11}\in S_0(\cE',\cE)$ and $\Phi_{22}\in S(\cU,\cY)$ arbitrary. Then $\fR_\Phi$ maps $S(\cE,\cE')$ onto the set $\{\Phi_{22}\}$. However, in case we combine the additional assumptions from Corollary \ref{C:specialcase1} and Proposition \ref{P:specialcase2}, $\fR_\Phi$ will  have this property.

\begin{proposition}\label{P:InjectClass}
Let $\Phi$ be as in \eqref{Phi}. Assume $\Phi$ is two-sided inner and $\Phi_{12}$ or $\Phi_{21}$ admits an $H^\infty$-inverse. Then for any $F,G\in S(\cE,\cE')$ with $\fR_\Phi[F]\simI \fR_\Phi[G]$ we have $F\simI G$.
\end{proposition}

It is convenient to first prove the following addition to Corollary \ref{C:confuncts}.

\begin{lemma}\label{L:LinfLeech}
Let $\Phi$ be as in \eqref{Phi} with $\Phi_{11}\in S_0(\cE',\cE)$ and let $K\in S(\cE,\cE')$. Then the functions $L$ and $L_*$ defined in Corollary \ref{C:confuncts} are defined a.e.\ on $\BT$ and when restricted to $\BT$ without loss of generality $L\in L^\infty(\cD_{\fR_\Phi[K]},\cD_K)$ and $L_*\in L^\infty(\cD_{\fR_\Phi[K]}^*,\cD_K^*)$.
\end{lemma}

\begin{proof}[\bf Proof.]
Since $\Phi_{11}\in S_0(\cE',\cE)$, the functions $(I-\Phi_{11}K)^{-1}$ and $(I-K\Phi_{11})^{-1}$ are $H^\infty$-functions as well. Set
\begin{align*}
N=D_{\fR_\Phi[K]},\quad M=D_K(I-\Phi_{11}K)^{-1}\Phi_{12},\\
N_*=D_{\fR_\Phi[K]^*},\quad M_*=D_{K^*}(I-\Phi_{11}K)^{-*}\Phi_{21}^*.
\end{align*}
Then $N$, $M$, $N_*$ and $M_*$ are $L^\infty$-functions, and the inequalities \eqref{Ineqs} translate to
\[
L_N^*L_N\geq L_M^*L_M\ands L_{N_*}^*L_{N_*}\geq L_{M_*}^*L_{M_*}.
\]
Thus, by Douglas lemma there exist $X\in\cL_1(L^2(\cD_{\fR_\Phi [K]}), L^2(\cD_K))$ and $Y\in\cL_1(L^2(\cD_K^*), L^2(\cD_{\fR_\Phi[K]}^*))$ such that $X L_N=L_M$ and $L_{N_*}^* Y=L_{M_*}^*$. It then follows by Proposition \ref{P:CLTapp} that without loss of generality $X=L_L$ and $Y=L_{L_*}$ for functions $L\in L^\infty_1(\cD_{\fR_\Phi[K]},\cD_K)$ and $L_*\in L^\infty_1(\cD_{\fR_\Phi[K]}^*,\cD_K^*)$.
\end{proof}

\begin{proof}[\bf Proof of Proposition \ref{P:InjectClass}.]
We first remark that our assumptions imply both $\Phi_{12}$ and $\Phi_{21}$ are a.e.\ invertible on $\BT$. Indeed, by Proposition \ref{P:specialcase2}, we know that $\Phi_{11}\in S_0(\cE',\cE)$ and $\Phi_{22}\in S_0(\cU,\cY)$. Set $\de=\min\{1-\|\Phi_{22}\|_\infty^2,\ 1-\|\Phi_{11}\|_\infty^2\}>0$.
Since $\Phi$ is two-sided inner, we have for a.e.\ $e^{it}\in\BT$ that
\begin{align*}
\Phi_{12}^*(e^{it})\Phi_{12}(e^{it})=D_{\Phi_{22}(\la)}^2\geq \de I\ands
\Phi_{12}(e^{it})\Phi_{12}^*(e^{it})=D_{\Phi_{11}^*(\la)}^2\geq \de I,\\
\Phi_{21}^*(e^{it})\Phi_{21}(e^{it})=D_{\Phi_{11}(\la)}^2\geq \de I\ands
\Phi_{21}(e^{it})\Phi_{21}^*(e^{it})=D_{\Phi_{22}^*(\la)}^2\geq \de I.
\end{align*}
Hence both $\Phi_{12}$ and $\Phi_{21}$ are invertible a.e.\ on $\BT$, as claimed.

Now let $F,G\in S(\cE,\cE')$ such that $\fR_\Phi[F]\simI \fR_\Phi[G]$.  Define $L$ and $L_*$ as in Corollary \ref{C:confuncts} with $K=G$ in the definition of $L$ and $K=F$ in the definition of $L_*$. Then, again using that $\Phi$ is two sided inner, we obtain that
\[
D_{\fR_\Phi[G]}=L^*D_{G}(I-\Phi_{11}G)^{-1}\Phi_{12} \ands
D_{\fR_\Phi[F]^*}=\Phi_{21}(I-F\Phi_{11})^{-1}D_{F^*}L_*^*
\]
holds a.e.\ on $\BT$.

Since $\fR_\Phi[F]\precI \fR_\Phi[G]$, there exists a function $\wtilQ\in L^\infty (\cD_{\fR_\Phi[G]},\cD_{\fR_\Phi[G]^*})$ such that
\begin{align*}
\fR_\Phi[F]-\fR_\Phi[G]&= D_{\fR_\Phi[F]^*} \wtilQ D_{\fR_\Phi[G]}\\
&=\Phi_{21}(I-F\Phi_{11})^{-1}D_{F^*}L_*^*\wtilQ L^*D_{G}(I-\Phi_{11}G)^{-1}\Phi_{12}.
\end{align*}
By Lemma \ref{L:EssComp}, we have
\[
\fR_\Phi[F]-\fR_\Phi[G]=\Phi_{21}(I-F\Phi_{11})^{-1}(F-G)(I-\Phi_{11}G)^{-1}\Phi_{12}.
\]
Now $\Phi_{21}$, $\Phi_{12}$, $(I-F\Phi_{11})^{-1}$ and $(I-\Phi_{11}G)^{-1}$ are all a.e.\ invertible on $\BD$, which yields
\[
F-G=D_{F^*}L_*^*\wtilQ L^*D_{G},\quad \mbox{a.e. on }\BT.
\]
By Lemma \ref{L:LinfLeech} the functions $L$ and $L_*$ are $L^\infty$-functions. Hence $Q=L_*^*\wtilQ L^*\in L^\infty(\cD_G,\cD_{F}^*)$. This yields that $F\simI G$, as functions in $L_1^\infty(\cE,\cE')$. However, this also implies that $F\simI G$, as functions in $S(\cE,\cE')$.
\end{proof}




\begin{thebibliography}{99}


\bibitem{AD08}
D.Z. Arov and H. Dym, {\em $J$-contractive matrix valued functions and related topics}, Encyclopedia of Mathematics and its Applications {\bf 116}, Cambridge University Press, Cambridge, 2008.


\bibitem{B95}
M. Bakonyi, A remark on Nehari's problem, {\em Integr.\ Equ.\ Oper.\ Theory} {\bf 22} (1995), 123--125.

\bibitem{B97}
A. Biswas, A harmonic-type maximal principle in commutant lifting, {\em Integr.\ Equ.\ Oper.\ Theory} {\bf 28} (1997), 373--381.

\bibitem{C76}
Z. Ceausescu, On intertwining dilations, {\em Acta Sci.\ Math.\ (Szeged)} {\bf 38} (1976), no. 3--4, 281-–290.

\bibitem{D66}
R.G. Douglas, On majorization, factorization, and range inclusion of operators on Hilbert space, {\em Proc.~Amer.~Math.~Soc.} {\bf 17} (1966), 413--415.

\bibitem{FF90}
C. Foias  and A.E. Frazho, \emph{The Commutant Lifting Approach to Interpolation Problems,}  Oper.\ Theory Adv.\ Appl.\ \textbf{44}, Birkh\"{a}user-Verlag, Basel, 1990.

\bibitem{FFGK98}
C. Foias, A.E. Frazho, I. Gohberg, and M.A. Kaashoek, {\em Metric Constrained Interpolation, Commutant Lifting and Systems}, Oper.\ Theory Adv.\ Appl.\ \textbf{100}, Birkh\"{a}user-Verlag, Basel, 1998.

\bibitem{F87}
B.A. Francis, {\em A Course in $H_{\infty}$ Control Theory}, Lecture Notes in Control and  Information Sciences \textbf{88}, Springer, Berlin, 1987.


\bibitem{GGK90}
I. Gohberg, S. Goldberg, and M.A. Kaashoek, {\em Classes of linear operators.\ Vol.\ I}, Oper.\ Theory Adv.\ Appl.\ {\bf 49}, Birkh\" auser Verlag, Basel, 1990.

\bibitem{tH11}
S. ter Horst, Redheffer representations and relaxed commutant lifting, {\em Complex Anal.\ Oper.\ Theory} {\bf 5} (2011), no. 4, 1051–-1072.

\bibitem{KSS91}
V.A. Khatskevich, Yu.L. Shmul'yan, and V.S. Shul'man, Pre-orders and equivalences in the operator ball (Russian), {\em Sibirsk.\ Mat.\ Zh.} {\bf 32} (1991), no.\ 3, 172--183; translation in {\em Siberian Math.\ J.} {\bf 32} (1991), no.\ 3, 496–-506 (1992).

\bibitem{NF70}
B. Sz.-Nagy and C. Foias, {\em Harmonic Analysis of Operators on Hilbert Space,}  North Holland Publishing Co., Amsterdam-Budapest, 1970.




\bibitem{R60}
R.M. Redheffer, On certain linear fractional representations, {\em J.\ Math.\ Phys.} {\bf 39} (1960), 269--286.

\bibitem{R62}
R.M. Redheffer, On the relation of transmission - line theory to scattering and transfer, {\em J.\ Math.\ Phys.} {\bf 41} (1962), 1--41.

\bibitem{S76}
Yu.L. Shmul'yan, Some stability properties for analytic operator-valued functions (Russian), {\em Mat.\ Zametki} {\bf 20} (1976), no. 4, 511-–520.

\bibitem{S78}
Yu.L. Shmul'yan, General linear-fractional transformations of operator spheres (Russian), {\em Sibirsk.\ Mat.\ Zh.} {\bf 19} (1978), no. 2, 418-–425.

\bibitem{S80}
Yu.L. Shmul'yan, Generalized fractional-linear transformations of operator spheres (Russian),
{\em Sibirsk.\ Mat.\ Zh.} {\bf 21} (1980), no.\ 5, 114-–131;  translation in {\em Siberian Math.\ J.} {\bf 21} (1991), no.\ 3, 496–-506 (1992).

\bibitem{iS72}
I. Suciu, Harnack inequalities for a functional calculus, In: {\em Hilbert space operators and operator algebras (Proc.\ Internat.\ Conf., Tihany, 1970)}, pp.\ 499–-511, {\em Colloq.\ Math.\ Soc.\ Janos Bolyai} {\bf 5}, North-Holland, Amsterdam, 1972.


\end{thebibliography}
\end{document}